\documentclass[reqno, 12pt, reqno]{amsart}

\usepackage{amsfonts, amsthm, amsmath, amssymb}
\usepackage{hyperref}
\hypersetup{colorlinks=false}

\usepackage[margin=1.5in]{geometry}

\RequirePackage{mathrsfs} \let\mathcal\mathscr
\numberwithin{equation}{section}

\usepackage[capitalise, noabbrev]{cleveref}

\newcommand{\Mod}[1]{\ (\mathrm{mod}\ #1)}
\newcommand{\op}[1]{\operatorname{#1}}
\newcommand{\mc}[1]{\mathcal{#1}}

\newcommand{\prim}{\textrm{prim}}

\renewcommand{\d}{\mathrm{d}}
\renewcommand{\phi}{\varphi}
\renewcommand{\rho}{\varrho}

\newcommand{\0}{\mathbf{0}}
\renewcommand{\P}{\mathbb{P}}
\newcommand{\Proj}{\P}

\newcommand{\F}{\mathbb{F}}
\newcommand{\Z}{\mathbb{Z}}

\newcommand{\N}{\mathbb{N}}
\newcommand{\Q}{\mathbb{Q}}
\newcommand{\R}{\mathbb{R}}

\newcommand{\OO}{\mathcal{O}}

\newcommand{\del}{\partial}
\newcommand{\isom}{\cong}

\renewcommand{\leq}{\leqslant}
\renewcommand{\geq}{\geqslant}

\newcommand{\bm}{\mathbf}

\newcommand{\bmm}{\mathbf{m}}

\newcommand{\bx}{\mathbf{x}}
\newcommand{\by}{\mathbf{y}}
\newcommand{\bc}{\mathbf{c}}
\newcommand{\bv}{\mathbf{v}}

\newcommand{\bz}{\mathbf{z}}
\newcommand{\bw}{\mathbf{w}}
\newcommand{\bb}{\mathbf{b}}
\newcommand{\ba}{\mathbf{a}}
\newcommand{\bk}{\mathbf{k}}
\newcommand{\bms}{\mathbf{s}}
\newcommand{\br}{\mathbf{r}}

\newcommand{\eps}{\epsilon}
\newcommand{\al}{\alpha}

\newcommand{\beps}{\boldsymbol{\eps}}

\DeclareMathOperator{\Pic}{Pic}

\newcommand{\supp}{\op{supp}}

\renewcommand{\hat}{\widehat}

\newcommand{\sqf}{\op{sqf}}

\newcommand{\1}{\mathbf{1}}

\newtheorem{theorem}{Theorem}[section]

\newtheorem{conjecture}[theorem]{Conjecture}

\newtheorem{proposition}[theorem]{Proposition}

\newtheorem{lemma}[theorem]{Lemma}

\theoremstyle{definition}

\newtheorem{definition}[theorem]{Definition}
\newtheorem{remark}[theorem]{Remark}

\newtheorem{example}[theorem]{Example}

\numberwithin{equation}{section}

\newcommand{\nid}{\noindent}

\newcommand{\ra}{\rightarrow}

\newcommand{\tensor}{\otimes}
\newcommand{\wt}{\widetilde{w}(\bx)}
\newcommand{\bs}{\backslash}

\newcommand{\PP}{\Proj}
\usepackage{tikz}
\usetikzlibrary{3d,calc}
\usetikzlibrary{positioning}

\begin{document}
\date{\today}

\title{Sums of Four Squareful numbers}
\maketitle
\begin{center}
\author{Alec Shute}\\
\vspace{7pt}
\address{Institute of Science and Technology\\ Am Campus 1, 3400 Klosterneuburg, Austria}\\
\email{alec.shute@ist.ac.at}
\end{center}

\begin{abstract}
    We find an asymptotic formula for the number of primitive vectors $(z_1,\ldots,z_4)\in (\Z_{\neq 0})^4$ such that $z_1,\ldots, z_4$ are all squareful and bounded by $B$, and $z_1+\cdots + z_4 = 0$. Our result agrees in the power of $B$ and $\log B$ with the Campana--Manin conjecture of Pieropan, Smeets, Tanimoto and V\'{a}rilly-Alvarado.
\end{abstract}

\tableofcontents{}

\section{Introduction}

The notion of \textit{Campana points}, first discussed by Campana in  \cite{campana2004orbifolds} and \cite{campana2011orbifoldes}, and Abramovich in \cite{abramovich2009birational}, is receiving increasing attention in the field of arithmetic geometry. Campana points associated to an orbifold $(X,D)$ can be viewed as rational points on $X$ that are integral with respect to a weighted boundary divisor $D$, and thus provide a way to interpolate between integral and rational points.  

The quantitative study of the arithmetic of Campana orbifolds was kick-started by the discussions in \cite{abramovich2009birational} and \cite{poonen2006projective}. A motivating example for the development of the theory was given in \cite{poonen2006projective}, where Poonen poses the problem of finding the number of coprime integers $z_0,z_1$ such that $z_0,z_1$ and $z_0+z_1$ are all squareful and bounded by $B$ (We recall that a nonzero integer $z$ is $m$-\textit{full} if for any prime $p|z$, we have $p^m|z$, and \textit{squareful} if it is $2$-full). This corresponds to counting Campana points on the orbifold $(\PP^1,D)$, where $D$ is the divisor $\frac{1}{2}[0]+\frac{1}{2}[1]+\frac{1}{2}[\infty]$. Since this problem seems too difficult at the moment, in \cite{van2012squareful}, Van Valckenborgh considers the higher-dimensional analogue $(\PP^{k-2},D)$, where $D = \sum_{i=1}^{k}\frac{1}{2}D_i$ for hyperplanes $D_1,\ldots,D_k$. This leads to the study of the asymptotic size $N_k(B)$ of the set
\begin{equation}\label{sum of k square-full numbers}
\mathcal{N}_k(B) = \left\{\bz\in  (\Z_{\neq 0})^k_{\prim} : |z_i|\leq B, z_i \textrm{ squareful for all } i, \sum_{i=1}^kz_i=0\right\},
\end{equation}
as $B \ra \infty$. The main result in \cite[Theorem 1.1]{van2012squareful} is that for any $k\geq 5$, we have $N_k(B) = cB^{k/2-1} + O(B^{k/2 -1- \delta})$ for some constants $c,\delta >0$. The remaining cases of interest are $k=3$, corresponding to Poonen's original question, and $k=4$. In this paper, we treat the case $k=4$. We define
\begin{equation}\label{the thin set}\mathcal{T} = \{(z_1, \ldots ,z_4) \in \Z^4_{\prim}: z_1\cdots z_4 = \square\},\end{equation} 
and consider the counting problem
\begin{equation}\label{main counting problem}
N(B) = \#(\mathcal{N}_4(B)\bs\mathcal{T}).
\end{equation}
Our main result is the following theorem.

\begin{theorem}\label{the main theorem} For any $\eps>0$ we have
    \begin{equation}
        N(B) = cB + O(B^{734/735 + \eps}),
    \end{equation}
for an absolute constant $c>0$ given explicitly in (\ref{the leading constant}) and Lemma \ref{local densities lemma}. 
\end{theorem}

Interest in the quantitative arithmetic of Campana orbifolds has been further sparked by the recent work \cite{pieropan2019campana}, in which Pieropan, Smeets, Tanimoto and V\'{a}rilly-Alvarado formulate a Manin-type conjecture for Fano Campana orbifolds. The authors establish their conjecture in the special case of vector group compactifications, using the height zeta function method developed by Chambert-Loir and Tschinkel in \cite{chambert2002distribution} and \cite{chambert2012integral}.  

In Section 2, we discuss the compatibility of Theorem \ref{the main theorem} with the prediction from \cite[Conjecture 1.1]{pieropan2019campana}, which is henceforth referred to as the \textit{Campana--Manin conjecture}. We find that our result agrees with the Campana--Manin conjecture in the power of $B$ and $\log B$, and the set $\mc{T}$ that we remove corresponds to a thin set of Campana points. The leading constant $c$ from Theorem \ref{the main theorem} is discussed in Section \ref{section proof of theorem 1.2}. However, in the light of our recent work\cite{leadingconstant2021}, we do not expect the leading constant to agree with the prediction from the Campana--Manin conjecture without the removal of further thin sets. 

A number of other instances of the Campana--Manin conjecture have been studied in addition to the results mentioned above. In \cite{browning2019arithmetic}, Browning and Yamagishi consider the orbifold $(\PP^n, D)$ for $D = \sum_{i=0}^{n+1}(1-\frac{1}{m_i})D_i$, where $D_0,\ldots, D_n$ are coordinate hyperplanes, and $D_{n+1}$ is a general hyperplane. Their main result is an asymptotic formula for the number of points on this orbifold under the assumptions that $m_i\geq 2$ for all $i$, and that there exists some $j\in \{0,\ldots, n+1\}$ such that
$$\sum_{\substack{0\leq i \leq n+1\\ i \neq j}}\frac{1}{m_i(m_i+1)}\geq 1.$$
In \cite{pieropan2020hyperbola}, Pieropan and Schindler establish the Campana--Manin conjecture for complete smooth split toric varieties satisfying an additional technical assumption, by developing a very general version of the hyperbola method. In \cite{xiao2020campana}, Xiao treats the case of biequivariant compactifications of the Heisenberg group over $\Q$, using the height zeta function method. Finally, in \cite{streeter2020campana}, Streeter studies $m$-full values of norm forms by considering the orbifold $(\PP_K^{d-1},(1-\frac{1}{m})\mathbb{V}(N_{E/K}))$, where $K$ is a number field, $\mathbb{V}(N_{E/K})$ the divisor cut out by a norm form associated to a degree-$d$ Galois extension $E/K$, and $m\geq 2$ is an integer which is coprime to $d$ if $d$ is not prime. 

We now summarise the proof of Theorem \ref{the main theorem}. The approach is broadly similar to previous results by Browning and Yamagishi in \cite{browning2019arithmetic} and Van Valckenborgh in \cite{van2012squareful}. Every nonzero squareful integer $z$ can be written uniquely in the form $z=y^3x^2$, where $y$ is a nonzero square-free integer and $x$ is a positive integer. For a fixed choice of $\by=(y_1,\ldots,y_4)$, the equation $z_1+\cdots +z_4=0$ defines a quadric $Q_{\by}$ cut out by the polynomial $F_{\by^3}(\bx) = \sum_{i=1}^4 y_i^3x_i^2$. The condition $z_1\cdots z_4 = \square$ from (\ref{the thin set}) is equivalent to the condition $y_1\cdots y_4 = \square$. Therefore
\begin{equation}\label{fibre sum}
N(B) = \frac{1}{16}\sum_{\substack{\by \in (\Z_{\neq 0})^4 \\y_1,\ldots, y_4 \textrm{ square-free}\\y_1\cdots y_4 \neq \square}}N_{\by}(B),
\end{equation}
where
$$N_{\by}(B) =\#\left\{\bx \in (\Z_{\neq 0})^4: F_{\by^3}(\bx)=0, \begin{tabular}{l}$\gcd(x_1y_1, \ldots, x_4y_4) = 1$\\
$\max_{1\leq i \leq 4}|y_i^3x_i^2| \leq B$\end{tabular}\right\}.$$
The factor $1/16$ comes from the fact that  for all $i\in\{1,\ldots,4\}$, there are two choices for the sign of $x_i$ corresponding to the same choice of $z_i$. 

In Section \ref{section the circle method}, we study the closely related counting problem 
\begin{equation}\label{Main counting problem, with coeffs}N_{\ba}(B) =\#\left\{\bx \in (\Z_{\neq 0})^4: F_{\ba}(\bx) = 0, \max_{1\leq i\leq 4}|a_ix_i^2| \leq B\right\},\end{equation}
where we have removed the coprimality condition and replaced $\by$ with an arbitrary vector $\ba = (a_1,\ldots, a_4) \in (\Z_{\neq 0})^4$.

\begin{remark} Rational points on non-singular quadric surfaces are known to conform to the classical Manin's conjecture \cite{franke1989rational}. Let $A = a_1\cdots a_4$. It can be shown that the Picard group of the quadric $F_{\ba}(\bx)=0$ has rank 1 when $A\neq \square$ and rank 2 when $A=\square$. Thus Manin's conjecture predicts that 
\begin{equation}\label{manin for quadrics}N_{\ba}(B)\sim\begin{cases}
c_{\ba}B, &\textrm{ if }A \neq \square\\
c'_{\ba}B\log B, &\textrm{ if }A=\square,
\end{cases}
\end{equation}
for constants $c_{\ba},c'_{\ba}$ (depending on $\ba$) which can be expressed as a product of local densities. The extra factor of $\log B$ in the case $A = \square$ explains why it is necessary to remove the set $\mc{T}$ in the formulation of Theorem \ref{the main theorem}.
\end{remark}

In order to estimate $N_{\ba}(B)$, we apply a version of the circle method introduced by Duke, Friedlander and Iwaniec in \cite{duke1993bounds}, and refined by Heath-Brown in \cite{MR1421949}. The more classical forms of the circle method used in \cite{browning2019arithmetic} and \cite{van2012squareful} are not sufficient for our purposes, but the version in \cite{MR1421949} is particularly well-suited to dealing with quadratic forms. In fact, in \cite{MR1421949}, Heath-Brown proves a vast generalisation of (\ref{manin for quadrics}) by finding an asymptotic formula for the quantity
$$N_w(Q,B) = \sum_{\substack{\bx \in \Z^n\\Q(\bx)=0}}w(B^{-1}\bx),$$
for any $n\geq 3$, any non-singular indefinite quadratic form $Q \in \Z[x_1,\ldots,x_n]$, and any $w: \R^n \ra \R_{\geq 0}$ belonging to a suitably defined class of smooth weight functions. The following theorem is the main result from Section \ref{section the circle method}. It features the quantities 
$$\Delta = \prod_{i=1}^4 \gcd\left(a_i, \prod_{j\neq i}a_j\right),\quad A=a_1\cdots a_4,$$
as well as the \textit{singular series} $\mathfrak{G}_{\ba}$ and the \textit{singular integral} $\sigma_{\infty}(\beps)$, defined respectively in (\ref{definition of weightless singular integral}) and (\ref{singular series definition}). Note that $\sigma_{\infty}(\beps)$ only depends on $\eps_i=\op{sgn}(a_i)$.
\begin{theorem}\label{main circle method result}
Let $\ba \in (\Z_{\neq 0})^4$ be such that $A\neq \square$ and $|A|\leq B^{4/7}$. Then
\begin{equation}
    N_{\ba}(B)= \frac{\mathfrak{G}_{\ba}\sigma_{\infty}(\beps)B}{|A|^{1/2}}+ O\left(\frac{B^{41/42+\eps}\Delta^{1/3}}{|A|^{11/24}}\right),
\end{equation}
where the implied constant depends only on $\eps$.
\end{theorem}
To complete the proof of Theorem \ref{the main theorem}, in Section \ref{section proof of theorem 1.2}, we apply Theorem \ref{main circle method result} together with an inclusion-exclusion argument in order to reinsert the coprimality condition and obtain an estimate for $N_{\by}(B)$. Returning to (\ref{fibre sum}), we can take a sum over these estimates for $N_{\by}(B)$ in the range $\max_i|y_i|\leq D$, where $D$ is a small power of $B$. The contribution from the remaining range $\max_{i}|y_i|>D$ is studied in Section \ref{section dealing with the large coefficients} using an elementary argument, and forms part of the error term in Theorem \ref{the main theorem}.

It is important in the proof of Theorem \ref{the main theorem} that the dependence of the estimate in Theorem \ref{main circle method result} on the coefficients $a_1,\ldots, a_4$ is made completely explicit.
 Several authors have obtained estimates of this type for quadratic forms. In \cite[Theorem 2]{browning2007density}, Browning applies the machinery from \cite{MR1421949} to find such an estimate for the counting problem corresponding to $N_{\ba}(B)$, but in $n\geq 5$ variables. Subsequently in \cite[Theorem 4.1]{browningbundle}, Browning and Heath-Brown carried this out in four variables. This latter result is nearly sufficient for our purposes, but Theorem \ref{main circle method result} represents a refinement in which we drop the assumptions made in \cite[Theorem 4.1]{browningbundle} that $A$ is nearly square-free and all the coefficients $a_1,\ldots, a_4$ are roughly the same size. Finally, we mention that using other techniques from the geometry of numbers, Comtat provides in \cite[Theorem 1.2]{comtat2019uniform} a completely uniform estimate for the number of zeros $(x_1,\ldots,x_n)\in \Z^n_{\prim}$  of a non-singular quadratic form $Q$ in $n\geq 3$ variables which lie in an arbitrary box $|x_i|\leq B_i$ for $i\in \{1,\ldots, n\}$. However, the resulting bound $N_{\ba}(B) \leq B/|A|^{1/4}$ does not have a good enough dependence on $A$ to be useful for our purposes. 

\subsection*{Notation} We take $\N = \Z_{\geq 1}$. We denote by $(\Z_{\neq 0})^n_{\prim}$ the set of vectors $(a_1,\ldots, a_n)$ such that $a_1,\ldots, a_n$ are nonzero integers and $\gcd(a_1,\ldots, a_n)=1$. For a prime $p$, we let $\nu_p$ denote the $p$-adic valuation. We write $e(\cdot)$ for the function $e^{2\pi i (\cdot)}$ and $e_q(\cdot)$ for the function $e^{2 \pi i (\cdot)/q}$. We use boldface letters to denote vectors with four components, for example $\bz = (z_1, \ldots, z_4)$. For a vector $\bv$, we define $|\bv| = \max_{1\leq i \leq 4}|v_i|$. All implied constants will be allowed to depend on $\eps$, but nothing else unless otherwise stated. Moreover, $\eps$ will denote a small positive number which for convenience we allow to take different values at different points in the argument. 
\newline
\hfill \break
\nid \textbf{Acknowledgements.} The author is grateful to Tim Browning for suggesting this project and for helpful feedback and guidance during the development of this work.

\section{The Campana--Manin conjecture}

In this section we recall from \cite[Section 3]{pieropan2019campana} the definition of Campana points and the statement of the Campana--Manin conjecture. We demonstrate that the set $\mc{T}$ removed in the definition of $N(B)$ is a thin set of Campana points, and that the power of $B$ and $\log B$ in Theorem \ref{the main theorem} is consistent with the prediction from the Campana--Manin conjecture. 

\begin{definition} Let $F$ be a field. A \textit{Campana orbifold} is a pair $(X,D)$, where $X$ is a smooth variety over $F$ and 
$$D = \sum_{\alpha \in \mathcal{A}}\eps_{\alpha}D_{\alpha}$$
is an effective Weil $\Q$-divisor of $X$ over $F$ (where the $D_{\alpha}$ are prime divisors) such that
\begin{enumerate}
    \item For all $\alpha \in \mathcal{A}$, either $\eps_{\alpha} =1$ or $\eps_{\alpha}$ takes the form $1-1/m_{\alpha}$ for some $m_{\alpha}\in \Z_{\geq 2}$.
\item The support $D_{\textrm{red}}=\sum_{\alpha \in \mathcal{A}}D_{\alpha}$ of $D$ has strict normal crossings on $X$.
\end{enumerate}
We say that a Campana orbifold is \textit{klt} if $\eps_{\alpha}\neq 1$ for all $\alpha \in \mathcal{A}$.
\end{definition}

Let $(X,D)$ be a Campana orbifold. Campana points will be defined as points $P\in X(F)$ satisfying certain conditions. These conditions are dependent on a finite set $S$ of places of $F$ containing all archimedean places, and a choice of \textit{good integral model} of $(X,D)$ over $\OO_{F,S}$. This model is defined to be a pair $(\mathcal{X},\mathcal{D})$, where $\mathcal{X}$ is a flat, proper model of $X$ over $\OO_{F,S}$, with $\mathcal{X}$ regular, and 
$$\mathcal{D} = \sum_{\alpha \in \mathcal{A}}\eps_{\alpha}\mathcal{D}_{\alpha},$$
where $\mathcal{D}_{\alpha}$ denotes the Zariski closure of $D_{\alpha}$ in $\mathcal{X}$. 

\begin{definition} Let $P \in (X\bs D_{\textrm{red}})(F)$. For a place $v \notin S$, let $\mathcal{P}_v$ denote the induced point in $\mathcal{X}(\OO_v)$ obtained via the valuative criterion for properness as stated in \cite[Theorem II.4.7]{hartshorne1977algebraic}. For $\alpha \in \mc{A}$, we define the \textit{intersection multiplicity} $n_v(\mathcal{D}_{\alpha},P)$ of $\mc{D}_{\al}$ and $P$ at $v$ to be the colength of the ideal $\mc{P}_v^{*}\mc{D}_{\al}$ in $\OO_v$. Then the \textit{intersection number} of $P$ and $\mathcal{D}$ at $v$ is defined to be 
$$n_v(\mathcal{D},P) = \sum_{\alpha \in \mc{A}}\eps_{\alpha}n_v(\mathcal{D}_{\alpha},P).$$
\end{definition}

\begin{definition} Let $(X,D)$ be a Campana orbifold with a good integral model $(\mc{X}, \mc{D})$ over $\OO_{F,S}$. A point $P \in (X\bs D_{\textrm{red}})(F)$ is a \textit{Campana} $\OO_{F,S}$\textit{-point} of $(\mc{X},\mc{D})$ if for all $v \notin S$ and all $\alpha \in \mc{A}$, we have
\begin{enumerate}
    \item If $\eps_{\alpha} = 1$, then $n_v(\mc{D}_{\al},P) =0$.
    \item If $\eps_{\al} \neq 1$, so that $\eps_{\al} = 1-1/m_{\al}$ for some $m_{\al} \in \Z_{\geq 2}$, then either $n_v(\mc{D}_{\al},P) =0$ or $n_v(\mc{D}_{\al},P)\geq m_{\al}$.
\end{enumerate}
We denote the set of Campana $\OO_{F,S}$-points of $(\mc{X},\mc{D})$ by $(\mc{X},\mc{D})(\OO_{F,S})$. 
\end{definition}

\example\label{example of campana orbifolds} Campana points are related to $m$-full values of polynomials. We consider a smooth projective variety $X\subseteq \PP^n$ over $\Q$, and a divisor
\begin{equation}\label{a concrete example}
D = \sum_{i=1}^k \left(1-\frac{1}{m_i}\right)D_i,
\end{equation}
where $m_i \geq 2$ are integers, and $D_i$ are prime divisors on $X$ cut out by polynomial equations $f_i=0$. Choosing the obvious good integral model $(\mc{X},\mc{D})$, a rational point $z \in (X\bs\bigcup_{i=1}^k D_i)(\Q)$, represented by $(z_0,\ldots, z_n)\in \Z^{n+1}_{\prim}$ is a Campana $\Z$-point of $(\mc{X},\mc{D})$ if and only if $f_i(z_0, \ldots, z_n)$ is $m_i$-full for all $i\in \{0,\ldots,k\}$.

\begin{definition}\label{what are campana thin sets} For an irreducible variety $X$ over $F$, a subset $A\subseteq X(F)$ is \textit{type I} if $A$ is a proper closed subvariety of $X$, and \textit{type II} if it can be written in the form $A=\varphi(V(F))$, where $V$ is an integral projective variety with $\dim(V) = \dim(X)$ and $\varphi\colon V \ra X $ is a generically surjective morphism of degree at least 2. A \textit{thin} set of $X(F)$ is a subset of a finite union of type I and type II sets. In \cite[Definition 3.7]{pieropan2019campana}, a thin set of Campana $\OO_{F,S}$-points is defined to be the intersection of a thin set of $X(F)$ with the set of Campana points $(\mathcal{X}, \mathcal{D})(\OO_{F,S})$.
\end{definition}

We now come to the statement of the Campana--Manin conjecture given in \cite[Conjecture 1.1]{pieropan2019campana}. Let $K$ be a number field, and let $(X,D)$ be a Campana orbifold over $K$ with a good integral model $(\mc{X},\mc{D})$ over $\OO_{K,S}$. Let $(\mc{L},\|\cdot\|)$ be an adelically metrized big and nef line bundle on $X$ and $[L]$ the associated divisor class. We let $H_{\mc{L}}\colon X(K) \ra \R_{\geq 0}$ denote the corresponding height function, as defined in \cite[Section 1]{peyre1995hauteurs}.
We recall that the \textit{effective cone} $\Lambda_{\op{eff}}$ of a variety $X$ is defined as
$$\Lambda_{\op{eff}} = \{[D]\in \op{Pic}(X): [D]\geq 0\}\tensor_{\Z}\R.$$

\begin{definition}\label{exponents a and b} Let $[K_X]$ denote the canonical divisor class. Given the above data, we define
$$ a = \inf\{t \in \R:  t[L]+ [K_X] + [D] \in \Lambda_{\op{eff}}\},$$
and we define $b$ to be the codimension of the minimal supported face of $\Lambda_{\op{eff}}$ which contains  $a[L] + [K_X] + [D]$. 
\end{definition}

\begin{conjecture}[Pieropan, Smeets, Tanimoto, V\'{a}rilly-Alvarado]\label{campana manin conjecture} Suppose that $(X,D)$ is a klt Campana orbifold, such that $-(K_X + D)$ is ample (in this case we say that the orbifold is \textit{Fano}). Assume that the set of Campana points $(\mc{X},\mc{D})(\OO_{K,S})$ is not itself thin. Then there is a thin set $\mc{T}$ of Campana $\OO_{K,S}$-points such that
$$\#\{P \in (\mc{X},\mc{D})(\OO_{K,S})\bs \mc{T}: H_{\mc{L}}(P) \leq B\} \sim cB^a(\log B)^{b-1},$$
as $B \ra \infty$, where $a,b$ are as in Definition \ref{exponents a and b}, and $c>0$ is an explicit constant described in \cite[Section 3.3]{pieropan2019campana}.
\end{conjecture}

\begin{remark}The hypothesis that the Campana points themselves are not thin is discussed by Nakahara and Streeter in \cite{samthincampana}. The authors establish in \cite[Theorem 1.1]{samthincampana} a connection between thin sets of Campana points and weak approximation, in the spirit of Serre's arguments in \cite[Theorem 3.5.7]{serretopicsgalois}. Combining this with \cite[Corollary 1.4]{samthincampana}, it can be shown that this hypothesis holds for the orbifold we consider below. 
\end{remark}
We now put the counting problem $N(B)$ from (\ref{main counting problem}) into the context of Conjecture \ref{campana manin conjecture}. The Campana orbifold under consideration is a special case of Example (\ref{example of campana orbifolds}). We consider the variety $X=\Proj^{2}$ over $\Q$. Below we will use $z_1,z_2,z_3$ to denote a representative of the point $[z_1:z_2:z_3] \in \PP^2(\Q)$ such that $(z_1,z_2,z_3) \in \Z^3_{\prim}$. Let $D$ be the divisor on $X$ given by
$$D = \sum_{i=1}^{4}\frac{1}{2}D_i,$$
where
$$D_i = \begin{cases} \{z_i = 0\},&\textrm{if } 1\leq i \leq 3,\\
\{z_1+z_2+z_3 = 0\},&\textrm{if } i=4.
\end{cases}
$$
The support of $D$ has strict normal crossings on $\Proj^2$, and so $(\Proj^2,D)$ is a Campana orbifold. Let $(\mathcal{X}, \mathcal{D})$ denote the obvious smooth proper model of $(\Proj^2,D)$ over $\Z$. Here we have chosen the set of bad places $S$ to consist only of the archimedean place $|\cdot |_{\infty}$. 

Let $z \in (\Proj^2 \bs D_{\textrm{red}})(\Q)$. The intersection multiplicity of $z$ and $D_i$ at a prime $p$ is then the $p$-adic valuation $\nu_p(z_i)$, when $i \in \{1, 2, 3\}$, and $\nu_p(z_1+ z_2+z_3)$ when $i=4$. The condition for a point $z \in (X \bs D_{\textrm{red}})(\Q)$ to be Campana $\Z$-point of $(\mc{X},\mc{D})$ is therefore that $z_1, z_2,z_3$ and $z_1+z_2+z_3$ are all squareful. 

For convenience we introduce a new variable $z_4$ which is defined by the equation $z_4 = z_1+ z_2+z_3$. We choose the ample line bundle $\mathcal{L}=\OO_{\PP^2}(1)$, and the generating set $\{z_1,z_2,z_3,z_4\}$ for the global sections of $\mc{L}$. This choice of generating set gives rise to the height function 
\begin{align*}
H: \PP^2 &\ra \R_{\geq 0}\\
z=[z_1:z_2:z_3] &\mapsto \max(|z_1|,|z_2|,|z_3|,|z_4|).
\end{align*}
We note that $(z_1,z_2,z_3) \in \Z^3_{\prim}$ if and only if $(z_1,\ldots, z_4) \in \Z^4_{\prim}$. Therefore, recalling the definition of $\mc{N}_4(B)$ from (\ref{sum of k square-full numbers}), there is a 2-1 map
\begin{equation}\label{a complicated map phi}
\begin{split}
    \phi:\mc{N}_4(B) &\ra \{z \in (\PP^2,\mc{D})(\Z): H(z) \leq B\}\\
    (z_1, \ldots, z_4)&\mapsto [z_1:z_2:z_3].
\end{split}
\end{equation}

We now compute the constants $a$ and $b$ from Conjecture \ref{campana manin conjecture} in this example. We let $[L]$ denote the hyperplane class corresponding to the line bundle $\mathcal{L}$, and $[K_{\PP^2}]$ the canonical divisor class. We recall that $\Pic \Proj^2 \isom \Z$, with the isomorphism given by the degree. We have $\deg([L]) =1$, $\deg([K_{\PP^2}]) = -3$ and $\deg([D]) = 4\cdot 1/2 = 2$. Therefore
\begin{align*}
        a &= \inf\{t \in \R:  t[L] + [K_{\PP^2}] + [D] \in \Lambda_{\op{eff}}\}\\
        &= \inf\{t \in \R:  t -3 +2 \geq 0\}\\
        &= 1.
\end{align*}
The minimal supported face of $\Lambda_{\op{eff}}$ which contains $a[L] + [K_{\PP^2}] + [D] = [0]$ is $\{0\}$, which has codimension one in $\Lambda_{\op{eff}}$, so $b = 1$. Conjecture \ref{campana manin conjecture} therefore states that there is a thin set of Campana points $\mc{Z} \subseteq (\PP^2,\mc{D})(\Z)$, such that
$$ \#\{P \in (\PP^2,\mc{D})(\Z)\bs \mc{Z}: H(P) \leq B\} = cB(1+o(1))$$
for some constant $c>0$. Theorem \ref{the main theorem} states that this result holds true when $\mc{Z} = \phi(\mc{T})$, where $\mc{T}$ is defined in (\ref{the thin set}), except possibly for a different value for the leading constant. To conclude this section, we show that $\phi(\mc{T})$ is indeed thin. 

\begin{lemma}\label{thin lemma} Let $\mc{T}$ and $\phi$ be defined as in (\ref{the thin set}) and (\ref{a complicated map phi}) respectively. Then $\phi(\mc{T})\cap (\PP^2,\mc{D})(\Z)$ is a thin set of Campana points in $(\PP^2,\mc{D})(\Z)$. 
\end{lemma}

\begin{proof} 
It suffices to show that $\phi(\mc{T})$ is a thin set in $\PP^2(\Q)$. By abuse of notation we view $\mc{T}$ as a subset of $\PP^3(\Q)$ via the map $(z_1,\ldots, z_4) \mapsto [z_1:\cdots: z_4]$. We begin by showing that $\mc{T}$ is a thin subset $\PP^3(\Q)$. Consider the weighted projective space $\Proj_{\Q}(2,1,1,1,1)$ with variables $t, z_1, \ldots, z_4$ (we refer the reader to \cite{hosgood2016introduction} for the basic definitions pertaining to weighted projective spaces). We have an embedding
\begin{align*}
    \nu:\PP(2,1,1,1,1) &\hookrightarrow \PP^{10}\\
    [t:z_1:\cdots : z_4] &\mapsto [t: z_1^2:z_1z_2:\cdots :z_4^2],
\end{align*}
which on the $z_i$-variables is the Veronese embedding of degree 2. The polynomial $f(t,z_1,\ldots,z_4) = t^2-z_1z_2z_3z_4$ is weighted homogeneous of degree $4$, and so defines a subvariety $V$ of $\PP(2,1,1,1,1)$. Let $Y$ denote the image of $\nu$ and write $t,y_{11},\ldots, y_{44}$ for variables on $\PP^{10}$. Then $\nu(V)$ is a hypersurface of $Y$ defined by the equation $t^2 = y_{12}y_{34}$. From this we see that $V$ is integral, projective and of dimension 3. 

Consider the morphism $\pi: V \ra \Proj^3$ defined by $[t: z_1: \cdots z_4]\ra  [z_1: \cdots: z_4]$. This is \'{e}tale of degree 2 on the open subset $V'$ of $V$ defined by $z_1\cdots z_4 \neq 0$. The set $W\subseteq \PP^3({\Q})$ defined by the equation $z_1\cdots z_4=0$ is a type I thin set, and $\mc{T}=\pi(V'(\Q))\cup W$, so we deduce that $\mc{T}$ is thin in $\PP^3_{\Q}$. 

We can now show that $\phi(\mathcal{T})$ is thin in $\Proj^2_{\Q}$. To do this, we intersect $\Proj^3,V'$ and $W$ with the hyperplane $H$ defined by the equation $z_1+z_2+z_3= z_4$. The map $\pi$ is ramified along the set $W$, which is a union of hyperplanes $H_i$ given by $\{z_i=0\}$. Since the intersection of $H$ with the $H_i$ is smooth and transversal, it follows from \cite[Section 9.4]{serre1989lectures} that $\pi((V'(\Q)\cap H(\Q))\cup (W\cap H(\Q))$ is thin in $H(\Q)$. The image of this set under the obvious isomorphism $H \isom \PP^2$ sending $[z_1:\cdots :z_4]$ to $[z_1:z_2:z_3]$ is precisely $\phi(\mc{T})$, so we conclude that $\phi(\mathcal{T})$ is a thin set in $\PP^2_{\Q}$. 
\end{proof}

\section{Dealing with the large coefficients}\label{section dealing with the large coefficients}

Given a nonzero squareful number $z_i$, we let $x_i,y_i$ denote the unique integers such that $x_i \in \N$, $y_i$ is square-free, and $z_i=y_i^3x_i^2$. We will also use the notation $Y=y_1\cdots y_4$. For $B,D\geq 1$, we define
\begin{align*}
M(B,D) &= \#\left\{\bz \in (\Z_{\neq 0})^4_{\prim}: 
\begin{tabular}{l}
$\sum_{i=1}^4 z_i =0, z_i \textrm{ squareful for all }i,$\\
$|\bz|\leq B, |Y|\geq D$\\
\end{tabular}
\right\}.
\end{align*}
The aim of this section is to prove the following upper bound.
\begin{proposition}\label{large coeffs are o(B)}
We have $M(B,D) = O(B^{1+\eps}D^{-1/12})$. 
\end{proposition}
The key result in the proof of Proposition \ref{large coeffs are o(B)} is the following upper bound for the quantity 
\begin{align*}
N(\bm{X},\bm{Y}) &= \#\left\{
\bx,\by \in (\Z_{\neq 0})^4: \begin{tabular}{l}$y_1,\ldots,y_4 \textrm{ square-free,} \sum_{i=1}^4 x_i^2y_i^3 = 0,$  \\
$|x_i|\leq X_i, |y_i|\leq Y_i \textrm{ for all }i$ \\
\end{tabular}
\right\}.
\end{align*}
\begin{proposition}\label{dealing with the large coeffs}We have 
$$N(\bm{X},\bm{Y}) = O((X_1\cdots X_4)^{1/2+\eps}(Y_1\cdots Y_4)^{2/3+\eps}).$$ 
\end{proposition}

We explain how to deduce Proposition \ref{large coeffs are o(B)} from Proposition \ref{dealing with the large coeffs}. We define 
\begin{align*}
M_1(B;\bm{R}) &= \#\left\{\bz \in (\Z_{\neq 0})^4: 
\begin{tabular}{l}
$\sum_{i=1}^4 z_i =0, z_i \textrm{ squareful for all }i,$\\
$|z_i|\leq B, R_i \leq |y_i| < 2R_i \textrm{ for all }i$\\
\end{tabular}
\right\}.
\end{align*}
Then
\begin{equation}\label{dyadic sum}
M(B,D) \ll \sum_{\substack{\bm{R} \textrm{ dyadic }\\R_1\cdots R_4 \geq D}}M_1(B;\bm{R}),
\end{equation}
We observe that the conditions $|\bz|\leq B$ and $R_i \leq |y_i|$ imply that $|x_i|^2\leq B/R_i^3$. Consequently,
$$ M_1(B;\bm{R}) \leq N\left(\left(\sqrt{\frac{B}{R_1^3}}, \ldots, \sqrt{\frac{B}{R_4^3}}\right),\left(2R_1, \ldots, 2R_4\right)\right).$$
Applying Proposition \ref{dealing with the large coeffs}, we obtain 
$$M_1(B;\bm{R}) \ll B^{1+\eps}(R_1\cdots R_4)^{-1/12}.$$
We conclude from (\ref{dyadic sum}) that 
\begin{align*}
    M(B,D)&\ll B^{1+\eps}\sum_{\substack{\bm{R} \textrm{ dyadic }\\ R_1\cdots R_4 \geq D}}(R_1\cdots R_4)^{-1/12}\ll B^{1+\eps}D^{-1/12},
\end{align*}
as claimed in Proposition \ref{large coeffs are o(B)}.

\begin{proof}[Proof of Proposition \ref{dealing with the large coeffs}]
For $k\ \in \{1,\ldots, 4\}$, we define
$$S_k(\alpha) = \sum_{\substack{x_k \in \Z_{\neq 0}\\|x_k|\leq X_k}}\sum_{\substack{y_k \in \Z_{\neq 0}\\|y_k|\leq Y_k\\y_k \textrm{ square-free}}}e(x_k^2y_k^3).$$
Then
$$N(\bm{X},\bm{Y}) = \int_{0}^1 \prod_{k=1}^4 S_k(\alpha) \mathrm{d}\alpha.$$
By H\"{o}lder's inequality, we have 
\begin{equation}\label{Holder ineq}
N(\bm{X},\bm{Y}) \leq \left(\prod_{k=1}^4 \int_{0}^1 |S_k(\alpha)|^4 \mathrm{d}\alpha\right)^{1/4}.\end{equation}
We fix $k \in \{1,\ldots, 4\}$, and to ease notation we write $X_k=X, Y_k=Y$. Then 
\begin{equation}\label{circle methody integral}
    \int_{0}^1 |S_k(\alpha)|^4 \mathrm{d}\alpha = N(X,Y),
\end{equation}
where
$$N(X,Y)=\#\left\{\bx,\by \in (\Z_{\neq 0})^4:
\begin{tabular}{l}
$y_1,\ldots, y_4\textrm{ square-free, }$   \\
$x_1^2y_1^3+x_2^2y_2^3 = x_3^2y_3^3+x_4^2y_4^3$ \\
 $|\bx|\leq X, |\by|\leq Y$
\end{tabular}
\right\}.\\$$
In view of (\ref{Holder ineq}), the task now is to show that 
\begin{equation}\label{the old prop}
N(X,Y) = O(X^{2+\eps}Y^{8/3+\eps}).
\end{equation}
Throughout the remainder of the argument, we make repeated use of the trivial estimate for the divisor function, namely that the number of divisors of a positive integer $d$ is $O(d^{\eps})$. 

We begin by considering the trivial cases $x_i = \pm x_j$ for some $i\neq j$. Without loss of generality, suppose $i=1, j=2$. We obtain
\begin{equation}\label{trivial cases}
   x_1^2(y_1^3 + y_2^3) = -(x_3^2y_3^3 + x_4^2 y_4^3). 
\end{equation}
If both sides of (\ref{trivial cases}) are zero, then  $y_1 = -y_2$, since we are assuming $x_1\neq 0$. Hence there are $O(XY)$ choices for $x_1,y_1,y_2$. On the right hand side, since $y_3, y_4$ are square-free, it follows that $x_3 = x_4$ and $y_3 = -y_4$, hence there are $O(XY)$ choices for $x_3,x_4,y_3,y_4$. This gives a total of $O(X^2Y^2)$ solutions. If both sides of the equation are nonzero, then there are $O(X^2Y^2)$ choices for $x_3,x_4,y_3,y_4$, and for any such choice, there are $O(X^{\eps}Y^{\eps})$ choices for $x_1,y_1,y_2$ by the trivial estimate for the divisor function. Overall, in the case $x_i = \pm x_j$, we conclude that there are $O(X^{2+\eps}Y^{2+\eps})$ solutions, which is satisfactory for establishing (\ref{the old prop}). From now on we assume $x_i \neq \pm x_j$ for all $i\neq j$. 

Returning to the integral representation of $N(X,Y)$ from (\ref{circle methody integral}), we can apply the Cauchy--Schwarz inequality in two different ways to $|S_k(\alpha)|^2$.  This gives the inequalities 
\begin{align*}
&|S_k(\alpha)|^2 \leq X\sum_{|x|\leq X}\sum_{\substack{|y_1|,|y_2|\leq Y\\ y_1, y_2 \textrm{ square-free }}}e(\alpha x^2(y_1^3-y_2^3)),\\
&|S_k(\alpha)|^2 \leq Y \sum_{\substack{|y|\leq Y\\y \textrm{ square-free }}}\sum_{|x_1|,|x_2|\leq X}e(\alpha y^3(x_1^2-x_2^2)).
\end{align*}
Applying these inequalities once each to $|S_k(\alpha)|^4$, we obtain 
\begin{equation}
\label{L(X,Y)}N(X,Y) \leq XYL'(X,Y),
\end{equation}
where
\begin{align*}
L'(X,Y) &= \#\left\{
\begin{tabular}{l}
$(x,x_1,x_2,y,y_1,y_2)\in (\Z_{\neq 0})^6:$  \\
$|x|,|x_1|,|x_2|\leq X,$ \\
$|y|,|y_1|,|y_2|\leq Y\textrm{ and }y,y_1,y_2 \textrm{ square-free, }$  \\
$x^2(y_1^3-y_2^3) = y^3(x_1^2-x_2^2).$
\end{tabular}
\right\}.
\end{align*}

It will be convenient to work with the quantity $L(X,Y)$ defined to be $L'(X,Y)$ but with the additional assumption $U/2<|x|\leq U$, which we denote by $|x|\sim U$. We will then perform a sum over dyadic intervals for $U\leq 2X$ at the end of the argument. We also need to extract from $L(X,Y)$ the greatest common divisor $k=\gcd(x,y)$. Note that since $y$ is square-free, so is $k$. Hence
$$L(X,Y) = \sum_{k\leq \min(U,Y)}\mu^2(k)L_{k}(X,Y),$$
where
\begin{align*}
L_{k}(X,Y) &= \#\left\{
\begin{tabular}{l}
$(x,x_1,x_2,y,y_1,y_2)\in (\Z_{\neq 0})^6:$  \\
$|x|\sim U/k,  |x_1|,|x_2|\leq X,$ \\
$|y|\leq Y/k,|y_1|,|y_2|\leq Y, \textrm{ and }y,y_1,y_2 \textrm{ square-free, }$  \\
$(x,y) = 1,$\\
$x^2(y_1^3-y_2^3) = ky^3(x_1^2-x_2^2).$
\end{tabular}
\right\}.
\end{align*}

We now obtain two different estimates for $L_k(X,Y)$, depending on the size of $k$. Our first estimate is suitable for large values of $k$.

We have $ky^3(x_1^2-x_2^2) \neq 0$, since  $y\neq 0$, and the cases $x_1 = \pm x_2$ have already been dealt with above. Consequently, for any fixed $x,y_1,y_2,k$, there are $O((XY)^{\eps})$ choices for $y,x_1,x_2$ such that $x^2(y_1^3-y_2^3) = ky^3(x_1^2-x_2^2)$, by the trivial estimate for the divisor function. Therefore

\begin{equation}\label{large k}
L_k(X,Y) \ll (XY)^{\eps} A(k),
\end{equation}
where
$$A(k) = \#\{|x|\sim U/k, |y_1|,|y_2|\leq Y: k|x^2(y_1^3-y_2^3)\}.$$ 

Now $k|x^2(y_1^3 - y_2^3)$ implies that there exists integers $d,r$ with $dr = k$, such that $d|x^2$ and $r|(y_1^3 - y_2^3)$. Observe that since $k$ is square-free,  $d|x^2$ if and only if $d|x$. Defining
\begin{align*} 
N(r) &= \#\{y_1, y_2 \leq Y: r|(y_1^3-y_2^3)\},
\end{align*}
we therefore have 
\begin{align}\label{A(k)}
A(k) &\leq \sum_{\substack{d,r\\ dr = k}}\sum_{\substack{|x| \sim U/k \\ d|x}}N(r)\ll U\sum_{\substack{d,r\\ dr= k}}\frac{N(r)}{dk}.
\end{align}

Note that $r\leq k\leq Y$, so $N(r) \ll Y^2\rho(r)/r^2$, where we have defined 
$$\rho(r) = \#\{\eta_1, \eta_2 \Mod{r}: \eta_1^3 \equiv \eta_2^3\}.$$
It is easy to see that $\rho(r)\ll r^{1+\eps}$. Indeed, for a prime $p$ and a fixed $\eta_1 \Mod{p}$, there are at most $3$ choices for $\eta_2 \Mod{p}$ such that $\eta_1^3 \equiv \eta_2^3 \Mod{p}$, and so $\rho(p) \leq 3p$. Since $r$ is square-free and $\rho$ is multiplicative, it follows from the Chinese Remainder theorem that 
$$\rho(r)= \prod_{p|r}\rho(p) \leq \prod_{p|r}3p \ll r^{1+\eps}.$$
Applying this bound on $\rho(r)$, we conclude that $N(r) \ll Y^2/r^{1-\eps} \ll Y^{2+\eps}/r$, and hence from (\ref{large k}) and (\ref{A(k)}) we obtain 
\begin{align}
L_k(X,Y) &\ll \frac{(XY)^{\eps}U}{k}\sum_{\substack{d,r\\ dr = k}}\frac{Y^{2+\eps}}{dr}\ll\frac{X^{\eps}Y^{2+\eps}U}{k^2}\label{large k estimate}. 
\end{align}

We now find a different estimate for when $k$ is small.  If $(x,y) = 1$ and $x^2(y_1^3-y_2^3) = ky^3(x_1^2-x_2^2)$, then $x^2|k(x_1^2-x_2^2)$, so we may define the integer $u = k(x_1^2 - x_2^2)/x^2$. Note $u\neq 0$ by the assumption $x_1 \neq \pm x_2$. We have $u \ll kX^2/(U/k)^2 = k^3X^2/U^2$. Therefore 

\begin{equation}
L_{k}(X,Y) \ll \sum_{u\ll \frac{k^3X^2}{U^2}}\sum_{\substack{|x|\sim U/k\\ |x_1|,|x_2| \leq X\\ k(x_1^2-x_2^2) = ux^2}}\sum_{\substack{|y| \leq Y/k\\ |y_1|, |y_2| \leq Y \\ uy^3  = y_1^3 - y_2^3}}1.\label{Lk(X,Y) large k}
\end{equation}

For a fixed $n = uy^3$, there are $O(Y^{\eps})$ solutions $y_1, y_2$ to $y_1^3-y_2^3=n$, using the trivial estimate for the divisor function, and so the inner sum of (\ref{Lk(X,Y) large k}) is $O(Y^{1+\eps}/k)$. Similarly, for a given $|x|\sim U/k$, there are $O(X^{\eps})$ choices for $x_1,x_2$ in the middle sum. Hence
$$L_{k}(X,Y) \ll (XY)^{\eps}\sum_{u \ll \frac{k^3X^2}{U^2}}\frac{UY}{k^2}\ll \frac{kX^{2+\eps}Y^{1+\eps}}{U}.$$ 

Combining this with the estimate in (\ref{large k estimate}), we obtain
\begin{align*}
    L_k(X,Y) &\ll X^{\eps}Y^{1+\eps}\min\left(\frac{UY}{k^2},\frac{kX^2}{U}\right)\\
    &\leq X^{\eps}Y^{1+\eps}\left( \frac{UY}{k^2}\right)^{2/3}\left(\frac{kX^2}{U}\right)^{1/3}\\
    &= \frac{X^{2/3+\eps}Y^{5/3+\eps}U^{1/3}}{k}.
\end{align*}
Finally, we take a sum over $k\leq Y$ and dyadic intervals $U\leq 2X$ to obtain 
\begin{align*}
L(X,Y) &\ll X^{2/3+\eps}Y^{5/3+\eps}\sum_{\substack{U \,\rm{ dyadic}\\ U\leq 2X}}\sum_{k \leq Y}\frac{U^{1/3}}{k}\\
&\ll X^{1+\eps}Y^{5/3+\eps}.
\end{align*}
Recalling \eqref{L(X,Y)}, we have established (\ref{the old prop}), which from (\ref{Holder ineq}) gives the required bound for $N(\bm{X},\bm{Y})$.
\end{proof}

\section{The circle method}\label{section the circle method}
In this section, we will use the circle method from \cite{MR1421949} to count zeros of diagonal quadratic forms in four variables, the main goal being the proof of Theorem \ref{main circle method result}. Before proceeding with the proof, we collect together some of the notation which we will use throughout this section. Much of the notation depends on a vector $\ba = (a_1,\ldots,a_4) \in (\Z_{\neq 0})^4$, which remains fixed throughout this section. 

\begin{itemize}
    \item $F_{\ba}(\bx)$ denotes the quadratic form $\sum_{i=1}^4 a_ix_i^2.$
    \item $\bm{P} = (P_1, \ldots, P_4)$, where $P_i = \sqrt{\frac{B}{|a_i|}}$ for $i\in \{1,\ldots, 4\}$.
    \item $N_{\ba}(B) =\#\left\{\bx \in (\Z_{\neq 0})^4: F_{\ba}(\bx) = 0, |x_i| \leq P_i\right\}$, as defined in (\ref{Main counting problem, with coeffs}).
    \item $A = a_1 \cdots a_4$.
    \item $\Delta = \Delta(\ba)=\prod_{i=1}^4\gcd\left(a_i,\prod_{j\neq i}a_j\right)$.
    \item $\beps = (\eps_1, \ldots, \eps_4)\in \{\pm 1\}^4$, where $\eps_i = \op{sgn}(a_i)= a_i/|a_i|$ is the sign of $a_i$.
    \item $G(\bx)$ is the quadratic form $ \sum_{i=1}^4 \eps_i x_i^2 = F_{\beps}(\bx)$.
    \item $\eta \in (0,1/4)$ is a small real parameter depending only on $\ba$ and $B$.
    \item $w:\R^4 \ra \R_{\geq 0}$ is an infinitely differentiable smooth weight function, which has compact support and satisfies $w(\bx) = 0$ for all $|\bx| \leq \eta$.
    \item $w_1,w_2$ are particular choices of such a smooth weight function, defined explicitly in (\ref{smooth weight w1 and w2}). 
    \item $Q = B^{1/2}$, and for a positive integer $q$, we write $r = q/Q$.
    \item $\bc$ denotes a vector in $\Z^4$, and we define $\bv = (v_1,\ldots, v_4)$ by $v_i = q^{-1}P_ic_i$.
    \item $C_i = \eta^{-1}B^{\eps}|a_i|^{1/2}$ for $i\in \{1,\ldots, 4\}.$
\end{itemize}
\begin{definition} The \textit{singular integral} $\sigma_{\infty}(\beps)$ and the \textit{singular series} $\mathfrak{G}_{\ba}$ associated to $F_{\ba}$ are given respectively by the equations
\begin{align}
\label{definition of weightless singular integral}\sigma_{\infty}(\beps) &= \int_{-\infty}^{\infty}\int_{[-1,1]^4}e(-\theta G(\bx)) \,\mathrm{d}\bx \,\mathrm{d}\theta,\\
  \mathfrak{G}_{\ba} &= \sum_{q=1}^{\infty}q^{-4}\sum_{\substack{k \bmod{q}\label{singular series definition}\\ \gcd(k,q)=1}} \sum_{\mathbf{b} \bmod{q}}e_q(kF_{\ba}(\mathbf{b})).
\end{align}
\end{definition}
For convenience we record again here the statement of Theorem \ref{main circle method result}. 
\begin{theorem}\label{theorem 1.4 again}
Let $\ba \in (\Z_{\neq 0})^4$ be such that $A\neq \square$ and $|A|\leq B^{4/7}$. Then
\begin{equation}
    N_{\ba}(B)= \frac{\mathfrak{G}_{\ba}\sigma_{\infty}(\beps)B}{|A|^{1/2}}+ O\left(\frac{B^{41/42+\eps}\Delta^{1/3}}{|A|^{11/24}}\right).
\end{equation}
\end{theorem}

The circle method from \cite{MR1421949} uses of smooth weight functions $w:\R^4\ra \R_{\geq 0}$, which we will take to approximate the characteristic function on $[-1,1]^4$. We introduce a smoothly weighted version of $N_{\ba}(B)$ given by  
\begin{equation}\label{Main problem, smooth weight version}
N_{w,\ba}(B)=\sum_{\substack{\bx \in (\Z_{\neq 0 })^4\\F_{\ba}(\bx) = 0}}w(P_1^{-1}x_1, \ldots, P_4^{-1}x_4).
\end{equation}
We also introduce a \textit{weighted singular integral} 
\begin{equation}
    \sigma_{\infty}(w) = \int_{-\infty}^{\infty}\int_{\R^4}w(\bx)e(-\theta G(\bx)) \,\mathrm{d}\bx \,\mathrm{d}\theta.\label{weighted singular integral definition}\\
\end{equation}
We will need to assume that $w$ is infinitely differentiable, has compact support, and vanishes in a neighborhood of the origin ($w(\bx) = 0 \textrm{ for } |\bx| \leq \eta$). Whilst the arguments in this section can be applied for any such choice of $w$, with implied constants depending on $w$, in order to get the power saving stated in Theorem \ref{the main theorem}, we will need to keep track of the dependence on $w$ in our estimates. Therefore, for a given $\eta >0$, we fix a particular choice of smooth weight functions $w_1, w_2: \R^4 \ra \R_{\geq 0}$. We recall the \textit{standard bump function} $\psi:\R^4 \ra \R_{\geq 0}$ is defined by 
\begin{equation}\label{standard bump}
    \psi(\bx) =
    \begin{cases}
    \exp(-(1-|\bx|^2)^{-1}), &\textrm{ if }|\bx|< 1,\\
    0, &\textrm{ otherwise.}
    \end{cases}
\end{equation}
We then define $w_1, w_2$ to be
\begin{equation}\label{smooth weight w1 and w2}
\begin{split}
  w_1(\bx) &= \begin{cases} 0, &\mbox{if } |\bx|\leq \eta,\\
e\psi\left(\frac{|\bx|}{\eta}-2\right), & \mbox{if } \eta < |\bx| \leq 2\eta,\\
1, & \mbox{if } 2\eta < |\bx| \leq 1-\eta,\\
e\psi\left(\frac{|\bx|}{\eta} - \frac{1-\eta}{\eta}\right), & \mbox{if } 1-\eta < |\bx| \leq 1,\\
0, & \mbox{if } |\bx|> 1,
\end{cases}\\ \\
w_2(\bx) &= \begin{cases} 0, &\mbox{if } |\bx|\leq \eta,\\
e\psi\left(\frac{|\bx|}{\eta}-2\right), & \mbox{if } \eta < |\bx| \leq 2\eta,\\
1, & \mbox{if } 2\eta < |\bx| \leq 1,\\
e\psi\left(\frac{|\bx|}{\eta} - \frac{1}{\eta}\right), & \mbox{if } 1 < |\bx| \leq 1+\eta,\\
0, & \mbox{if } |\bx|> 1+\eta.
\end{cases}
\end{split}
\end{equation}
For any integers $j_1, \ldots, j_4 \geq 0$ with $j_1+\cdots  +j_4\leq N$, and for any $i\in\{1,2\}$, we have 
\begin{equation}\label{derivatives of smooth weights}\frac{\del^{j_1+\cdots +j_4}}{\del x_1^{j_1}\cdots \del x_4^{j_4}}w_i(\bx) \ll_N \eta^{-N}.\end{equation}

We now state a smoothly weighted version of Theorem \ref{theorem 1.4 again}. 
\begin{theorem}\label{the main smooth weight result} Suppose $\eta \in (0,1/4)$ and $\ba \in (\Z_{\neq 0})^4$. Let $w$ be one of the weights $w_1$ or $w_2$ from (\ref{smooth weight w1 and w2}). Suppose that $A\neq \square$. Then 
    \begin{equation}\label{smooth weight error terms}\begin{split}
    N_{w,\ba}(B) &= \frac{\mathfrak{G}_{\ba}\sigma_{\infty}(w)B}{|A|^{1/2}}+ O\left(\frac{\eta^{-6}B^{5/6+\eps}\Delta^{1/3}}{|A|^{5/24}}\right)+O(\eta^{-7}B^{1/2+\eps}).
    \end{split}\end{equation}
\end{theorem}

We explain how Theorem \ref{theorem 1.4 again} can be deduced from Theorem \ref{the main smooth weight result} by applying the methods from \cite[Section 5.3]{BHBquadricsurfaces} and \cite[Section 2.4]{MR3944106}, together with the upper bound $\mathfrak{G}_{\ba}\ll |A|^{\eps}\Delta^{1/4}$ for the singular series, which we prove in Lemma \ref{the one and only singular series estimate} at the end of this section. 

Let $1_{[-1,1]^4}$ denote the characteristic function on $[-1,1]^4$, and for an integer $j\geq 0$, define $w_2^{(j)}(\bx) = w_2((2\eta)^{-j}\bx)$. Then for any $\bx \in \R^4$, we have 
$$w_1(\bx) \leq 1_{[-1,1]^4}(\bx) \leq \sum_{j=0}^{\infty}w_2^{(j)}(\bx),$$
and consequently, 
\begin{equation*}
    N_{w_1,\ba}(B) \leq N_{\ba}(B) \leq \sum_{j=0}^{\infty}
    N_{w_2^{(j)},\ba}(B)=\sum_{j=0}^{\infty}N_{w_2,\ba}((2\eta)^jB).
\end{equation*}
The assumption $|A|\leq B^{4/7}$ in the statement of Theorem \ref{theorem 1.4 again} implies that the error term $\eta^{-7}B^{1/2+\eps}$ is dominated by the other error term in (\ref{smooth weight error terms}), provided that $\eta \gg B^{-5/14+\eps}$. For any $\eta \in (0,1/4)$ satisfying this condition, it follows from Theorem \ref{the main smooth weight result} that 
$$\sum_{j=0}^{\infty}N_{w_2,\ba}((2\eta)^jB) = \frac{(1+O(\eta))\mathfrak{G}_{\ba}\sigma_{\infty}(w_2)B}{|A|^{1/2}}+ O\left(\frac{\eta^{-6}B^{5/6+\eps}\Delta^{1/3}}{|A|^{5/24}}\right).$$
As explained in \cite[Lemma 2.9]{MR3944106}, for $i\in\{1,2\}$, we have 
$$|\sigma_{\infty}(w_i) - \sigma_{\infty}(\beps)|\ll \eta \sigma_{\infty}(\beps) \ll \eta,$$
from which we deduce that 
$$N_{\ba}(B) = \frac{(1+O(\eta))\mathfrak{G}_{\ba}\sigma_{\infty}(\beps)B}{|A|^{1/2}}+O\left(\frac{\eta^{-6}B^{5/6+\eps}\Delta^{1/3}}{|A|^{5/24}}\right).$$
We choose 
$$\eta = \frac{1}{5}B^{-1/42}|A|^{1/24}.$$
Clearly $\eta \gg B^{-5/14+\eps}$. Moreover,
using the assumption $|A|\leq B^{4/7}$, we have that $\eta \in (0,1/4)$, as was required in order to apply Theorem \ref{the main smooth weight result}. Theorem \ref{theorem 1.4 again} now follows from the estimate $\mathfrak{G}_{\ba}\ll |A|^{\eps}\Delta^{1/4}$ found in Lemma \ref{the one and only singular series estimate}.\medskip

We commence with the proof of Theorem \ref{the main smooth weight result}. It will be convenient to make the assumptions 
\begin{equation}\label{harmless assumption on eta}\eta^{-1}, |\ba| \ll  B^{R}\end{equation}
for some fixed $R \geq 0$, so that quantities bounded by an arbitrarily small power of $\eta^{-1}$ or $|A|$ are also $\ll B^{\eps}$ for any $\eps >0$. If these assumptions do not hold, then the statement of Theorem \ref{the main smooth weight result} is trivial. Employing the machinery from \cite[Theorem 2]{MR1421949}, we have
\begin{equation}\label{NwP}N_{w,\ba}(B) = \frac{C_Q}{Q^2}\sum_{\bc\in \Z^4}\sum_{q=1}^{\infty}q^{-4}S_{q,\ba}(\bc)I_{q,\ba}(\bc),\end{equation}
where $C_Q = 1+ O_N(Q^{-N})$, and $Q = B^{1/2}$. Here $I_{q,\ba}(\bc)$ and $S_{q,\ba}(\bc)$ are exponential sums and integrals defined analogously to \cite{MR1421949}, via
\begin{equation}\label{singular definitions}
\begin{split}
    S_{q,\ba}(\bc) &= \sum_{\substack{k \bmod q\\ \gcd(k,q)=1}} \sum_{\mathbf{b} \bmod q}e_q(kF_{\ba}(\mathbf{b}) + \mathbf{b}\cdot \bc),\\
    I_{q,\ba}(\bc) &= \int_{\R^4}w(P_1^{-1}x_1, \ldots, P_4^{-1}x_4)h\left(\frac{q}{Q}, \frac{F_{\ba}(\bx)}{Q^2}\right)e_q(-\bc\cdot \bx)\,\mathrm{d}\bx.
    \end{split}
\end{equation}
The smooth function $h:(0,\infty)\times \R \ra \R$ is as defined in \cite[Section 3]{MR1421949} and is given by
\begin{equation}\label{definition of h(x,y)}
h(x,y) = \sum_{j=1}^{\infty}(xj)^{-1}\left(u(xj)-u(|y|/xj)\right),
\end{equation}
where $u = c\psi(4x-3)$ is a linear transformation of the standard bump function (the one-dimensional analogue of (\ref{standard bump})), scaled by a constant $c$ so that its integral over $\R$ is equal to 1. As observed in \cite[Section 3]{MR1421949}, it is straightforward to check from (\ref{definition of h(x,y)}) that $h(x,y)\ll x^{-1}$ and $h(x,y)=0$ whenever $x> \max(1,2|y|)$.    

We note that by changing variables $P_i^{-1}x_i$ into $x_i$, we can rewrite $I_{q,\ba}(\bc)$ as 
\begin{align}
    I_{q,\ba}(\bc)&=P_1\cdots P_4\int_{\R^4}w(\bx)h\left(\frac{q}{Q}, \frac{F_{\ba}(P_1x_1, \ldots, P_4x_4)}{Q^2}\right)e(-\bv\cdot\bx)\mathrm{d}\bx\nonumber\\
    &= P_1\cdots P_4\int_{\R^4}w(\bx)h(r,G(\bx))e(-\bv\cdot\bx) \mathrm{d}\bx.\label{iqa0}
\end{align}
\subsection{The main term}

The main term for $N_{w,\ba}(B)$ comes from the case $\bc = \0$ in \eqref{NwP}. We define
\begin{equation}\label{definition of M(B)}
M(B) = \frac{1}{Q^2}\sum_{q=1}^{\infty}q^{-4}S_{q,\ba}(\0)I_{q,\ba}(\0).
\end{equation} 
We have 
\begin{equation}\label{singular series} \mathfrak{G}_{\ba} =\sum_{q=1}^{\infty}q^{-4}S_{q,\ba}(\0),\end{equation}
and from (\ref{iqa0}), 
$$I_{q,\ba}(\0) = P_1\cdots P_4\int_{\R^4}w(\bx)h\left(r, G(\bx)\right)\mathrm{d}\bx.$$

We let $w_0$ denote a smooth weight function supported on $[-17, 17]$ and taking the value 1 in $[-16, 16]$. Since $\supp(w)\subseteq [-1-\eta, 1+\eta]\subset [-2,2]$, we have $w_0(G(\bx))=1$ whenever $\bx \in \supp(w)$. We obtain
$$I_{q,\ba}(\0) = P_1\cdots P_4\int_{\R^4}w(\bx)w_0(G(\bx))h(r,G(\bx)) \mathrm{d}\bx.$$

The arguments at the beginning of \cite[Section 4.3]{browningbundle} can be applied to obtain
\begin{equation}\label{Main term J bit}
I_{q,\ba}(\0) = P_1\cdots P_4 \left( \lim_{\delta\ra 0}\int_{-\infty}^{\infty}\left(\frac{\sin(\pi \delta \theta)}{\pi \delta \theta}\right)^2J(w;\theta)L(\theta) \,\mathrm{d}\theta \right),
\end{equation}
where 
\begin{align}
    J(w;\theta) &= \int_{\R^4}w(\bx)e(-\theta G(\bx))\,\mathrm{d}\bx,\\
    L(\theta) &= \int_{-\infty}^{\infty}w_0(t)h(r,t)e(\theta t)\, \mathrm{d}t. 
\end{align}
We would like to compare the bracketed expression in (\ref{Main term J bit}) with the weighted singular integral $\sigma_{\infty}(w)$ from (\ref{weighted singular integral definition}), which by definition equals $\int_{-\infty}^{\infty}J(w;\theta)\mathrm{d}\theta$. We apply \cite[Lemma 4.11]{browningbundle} to estimate $L(\theta)$. In addition to $N$, the implied constants in \cite[Lemma 4.11]{browningbundle} depend only on $\mathrm{supp}(w_0)$. However, we note that we have chosen $w_0$ independently of $w$, and hence independently of $\eta$. Therefore for any $q \leq Q$ and $N\geq 1$, we have
\begin{equation}\label{estimate for L theta} L(\theta) = 1 + O_N(\{1+|\theta|^N\}r^N).\end{equation}

Our next task is to bound $J(w;\theta)$. In order to achieve this, we need to work with more general smooth weight functions $\widetilde{w}:\R^4\ra \R_{\geq 0}$ which belong to a class $S$ defined by the following properties.
\begin{definition}\label{class of smooth weights S}
We say $\widetilde{w}\in S$ if
\begin{enumerate}
    \item $\widetilde{w}$ is smooth,
    \item $\widetilde{w}$ is supported on $[-2,2]^4$,
    \item $\wt = 0$ for $|\bx| \leq \eta,$
    \item $|\wt|\ll 1$ for all $\bx \in \R^4$,
    \item The derivatives of $\wt$ are bounded as in (\ref{derivatives of smooth weights}).
\end{enumerate}\end{definition}
\nid Since $w$ is equal to $w_1$ or $w_2$, we note in particular that $w \in S$.
\begin{lemma}\label{estimate for J} For any $\widetilde{w}\in S$ and any $M\in \Z_{\geq 0}$, we have $J(\widetilde{w};\theta) \ll_M |\eta^2 \theta|^{-M}.$
\end{lemma}
\begin{proof} The case $M=0$ is trivial since the integrand in the definition of $J(\widetilde{w};\theta)$ is bounded and has compact support. For $M\geq 1$, we follow a similar argument to \cite[Lemma 10]{MR1421949}, by applying integration by parts and induction on $M$. We have 
$$J(\widetilde{w};\theta) = \int_{\R^4}\wt e(-\theta G(\bx)) \mathrm{d}\bx \ll \max_{1 \leq j \leq 4}\left|\int_{\substack{\bx \in \R^4,\\|x_j| = \max_i|x_i|}}\wt e(-\theta G(\bx)) \mathrm{d}\bx\right|.$$
Without loss of generality, we may assume that $j=1$. We note that 
$$e(-\theta G(\bx)) = \frac{1}{-4\pi i \eps_1\theta x_1}\frac{\del}{\del x_1}e(-\theta G(\bx)). $$
We perform integration by parts with respect to $x_1$ to obtain
\begin{align*}
J(\widetilde{w};\theta) \ll \left|\int_{\substack{\bx \in \supp{w}\\|x_1|= \max_i|x_i|}}\frac{\del}{\del x_1}\left(\frac{\wt}{4\pi i \theta x_1}\right)e(-\theta G(\bx))\mathrm{d}\bx\right|.
\end{align*}

Since $|x_1| = \max_i|x_i|$, we have $|x_1|^{-1}= |\bx|^{-1} \leq \eta^{-1}$ for any $\bx \in \supp{w}$. Therefore by the assumptions on the size of $\widetilde{w}$ and its derivatives, we see that 
$$\frac{\del}{\del x_1}\left(\frac{\wt}{4\pi i \theta x_1}\right) = \widetilde{w}^{(1)}(4\pi i\theta\eta^2)^{-1}$$
for some $\widetilde{w}^{(1)}\in S$. Hence 
$$J(\widetilde{w};\theta)\ll |\eta^2\theta|^{-1}|J(\widetilde{w}^{(1)};\theta)|.$$
Proceeding by induction, we obtain $J(\wt;\theta)\ll_M |\eta^2\theta|^{-M}|J(\widetilde{w}^{(M)};\theta)|$ for some $\widetilde{w}^{(M)}\in S$, from which we deduce the required result by applying the trivial bound $J(\widetilde{w}^{(M)};\theta) \ll 1$.
\end{proof}

From Lemma \ref{estimate for J}, for any integer $M\geq 1$, we have
$$J(w;\theta) \ll_M (1+|\eta^2\theta|)^{-M}\leq \eta^{-2M}(1+|\theta|)^{-M}.$$
Combining this with (\ref{Main term J bit}), (\ref{estimate for L theta}) and the fact that  
$$\lim_{\delta \ra 0}\int_{-\infty}^{\infty}\left(\frac{\sin(\pi \delta\theta)}{\pi \delta\theta}\right)^2J(\theta;w)\op{d}\theta =\int_{-\infty}^{\infty}J(\theta;w)\op{d}\theta =\sigma_{\infty}(w), $$ 
we obtain
$$|(P_1\cdots P_4)^{-1}I_{q, \ba}(\0)-\sigma_{\infty}(w)|\ll_{M,N} \int_{-\infty}^{\infty}\frac{r^N\{1+|\theta|\}^N}{\eta^{2M}\{1+|\theta|\}^{M}}\,\mathrm{d}\theta.$$ 
In order to ensure that the integral converges, we make the choice $M = N+2$, and we conclude that 
\begin{equation}\label{IQnaught}I_{q,\ba}(\0) = P_1\cdots P_4 \{\sigma_{\infty}(w) + O_N(\eta^{-2N-4}r^N)\},\end{equation}
for any integer $N\geq 1$. 

Let $R= B^{1/2-\eps}\eta^2$. We note that if $q\leq R$, then the error term in (\ref{IQnaught}) can be made smaller than any negative power of $B$ by appropriate choice of $N$ (depending on $\eps$), due to the assumption in (\ref{harmless assumption on eta}).

We now split the main term up. We have 
\begin{equation}\label{split the main term up}
\begin{split}
M(B) = \frac{P_1\cdots P_4\sigma_{\infty}(w)}{B}\sum_{q\leq R}q^{-4}S_{q,\ba}(\0)+O(T(R)+1),
\end{split}
\end{equation}
where 
\begin{equation}\label{TR}T(R) = \frac{1}{Q^2}\sum_{q>R}q^{-4}S_{q,\ba}(\0)I_{q,\ba}(\0).\end{equation} 
We will estimate $T(R)$ using partial summation. The following lemma is similar to \cite[Lemma 4.3]{browningbundle}, and gives bounds for $I_{q,\ba}(\bc)$ and its derivative. Only the case $\bc =\0$ is needed in the study of the main term, but the case $\bc \neq \0$ will be useful later.
\begin{lemma}\label{lemma 4.3*} Let $\bc \in \Z^4$ and $k\in \{0,1\}$. Then
\begin{enumerate}
\item If $k=0$ or $\bc = \0$ then $q^k\frac{\del^kI_{q,\ba}(\bc)}{\del q^k} \ll P_1\cdots P_4$,
\item If $k=1$ and $\bc \neq \0$ then $q^k\frac{\del^kI_{q,\ba}(\bc)}{\del q^k} \ll \eta^{-1}P_1\cdots P_4$.
\end{enumerate}
\end{lemma}

\begin{proof} Suppose that $\bc=\0$. Recalling (\ref{iqa0}) and the notation $r=q/Q$, it is clear that 
\begin{equation}\label{derivatives of iqn}q^k\frac{\del^kI_{q,\ba}(\0)}{\del q^k} = r^{k}P_1\cdots P_4\int_{\R^4}w(\bx)\frac{\del^k h(r,G(\bx))}{\del r^k} \mathrm{d}\bx.\end{equation}
From \cite[Lemma 5]{MR1421949} with $m=k, n=0$ and $N=2$, we have 
$$\frac{\del^k h(r,G(\bx))}{\del r^k}\ll r^{-1-k}\left(r^2+ \min\left\{1,\frac{r^2}{|G(\bx)|^2}\right\}\right),$$
and so 
\begin{equation}\label{measures of sets}(P_1\cdots P_4)^{-1}\frac{\del^kI_{q,\ba}(\bc)}{\del q^k} \ll r + r^{-1}\int_{\substack{\bx \in \supp w\\ |G(\bx)|\leq r}}1\mathrm{d}\bx + r\int_{\substack{\bx \in \supp w\\|G(\bx)|>r}}\frac{1}{|G(\bx)|^2}\mathrm{d}\bx. \end{equation}

We need to show that the right hand side of (\ref{measures of sets}) is $O(1)$. We can assume that the first term $r$ is $O(1)$, since $h(r,G(\bx))=0$ when $r>\max(1,|G(\bx)|)$, and $G(\bx) \ll 1$ for $x \in \supp w$, and so $I_{q,\ba}(\0)=0$ unless $r \ll 1$. In order to estimate the integrals in (\ref{measures of sets}), we consider for $z>0$ the Lebesgue measure $m(z)$ of the set 
$$S(z)=\{\bx \in \supp w:|G(\bx)|\leq z\}.$$ 
We can find distinct indices $i,j$ such that $\eps_i = \eps_j$. Without loss of generality we may assume $i=1, j=2$, and in addition that $\eps_1 = \eps_2 = 1$. For a fixed choice of $x_3, x_4$, we define $c = \eps_3x_3^2 + \eps_4x_4^2$, so that if $\bx \in S(z)$ then $x_1^2+x_2^2 \in [-c-z,-c+z]$. The measure of the set of pairs $x_1,x_2$ satisfying this condition is $O(z)$. For $\bx \in \supp w$, we require $|x_3|, |x_4| \ll 1$, and hence $m(z) = O(z)$. From this we see that the first integral in (\ref{measures of sets}) is $O(r)$, and by breaking into dyadic intervals, the second integral is $O(r^{-1})$. Therefore the right hand side of (\ref{measures of sets}) is $O(1)$, as required.

When $\bc \neq \0$, we first note that for $k=0$, we can reduce to the case above by applying the trivial estimate to the extra exponential factor $e(-\bv\cdot\bx)$ appearing in the integral defining $I_{q,\ba}(\bc)$. It remains only to deal with the case $\bc \neq \0, k=1$.  Define the vector $\bv \in \R^4$ by $v_i = q^{-1}P_ic_i$. Similarly to (\ref{derivatives of iqn}), we obtain
\begin{align*}
&(P_1 \cdots P_4)^{-1}q\frac{\del I_{q,\ba}(\bc)}{\del q} \ll r\int_{\R^4}w(\bx)\frac{\del \{h(r,G(\bx))e(-\bv\cdot \bx)\}}{\del r} \mathrm{d}\bx\\
& = r\int_{\R^4}w(\bx)\left(\frac{\del h(r,G(\bx))}{\del r}e(-\bv\cdot \bx) + h(r,G(\bx))\frac{-2\pi i \bv \cdot \bx }{r}e(-\bv\cdot \bx)\right) \mathrm{d}\bx.
\end{align*}
The first term can again be dealt with using the trivial estimate for $e(-\bv\cdot \bx)$ and the argument for the case $\bc = \0$ given above. In order to estimate the second term, as in \cite[Lemma 14]{MR1421949}, we apply the divergence theorem 
$$\int_{\R^4}\nabla\cdot (w(\bx)h(r,G(\bx)e(-\bv\cdot \bx)\bx)\mathrm{d}\bx = 0$$ 
in order to remove the additional factor of $2\pi i \bv \cdot \bx$. This yields 
\begin{equation}\label{divergence theorem}
\begin{split}
&\int_{\R^4}w(\bx)h(r,G(\bx))e(-\bv\cdot \bx)(2\pi i \bv \cdot \bx)\mathrm{d}\bx =\int_{\R^4}\wt h(r,G(\bx))e(-\bv \cdot \bx)\mathrm{d}\bx\\
 +&\int_{\R^4}4w(\bx) h(r,G(\bx))e(-\bv\cdot \bx)\mathrm{d}\bx+\int_{\R^4}2w(\bx)G(\bx)e(-\bv\cdot \bx)\frac{\del h(r,G(\bx))}{\del G(\bx)}\mathrm{d}\bx,
\end{split}
\end{equation}
where $\wt = (\bx \cdot \nabla )w(\bx)$. The second integral on the right hand side of (\ref{divergence theorem}) can now be treated in the same way as the easier cases $k=0$ or $\bc = \0$ using the trivial estimate for $e(-\bv \cdot \bx)$. The third integral can similarly be estimated by noting that $G(\bx)\ll 1$ for $\bx \in \supp(w)$, and applying \cite[Lemma 5]{MR1421949} with $n=1,m=0,N=0$. For the first integral, we note that $\wt$ is uniformly bounded by $\eta^{-1}$. Hence the above arguments can be applied to the first integral as well, but with an extra factor of $\eta^{-1}$.
\end{proof}

We define 
\begin{equation}\label{Sigma bc}\Sigma(x;\bc) = \sum_{q\leq x}q^{-3}S_{q,\ba}(\bc).\end{equation}
Moreover, we let $F^*_{\ba}$ denote the \textrm{dual form} of $F_{\ba}$, i.e., the quadratic form given by the equation 
$$F^*_{\ba}(\bc) = a_2a_3a_4c_1^2+a_1a_3a_4c_2^2+a_1a_2a_4c_3^2+a_1a_2a_3c_4^2.$$
We also define 
\begin{equation}\label{sneaky delta def}
\Delta_{\bc}(\ba) = \prod_{i=1}^4\gcd\left(\gcd(a_i,c_i), \prod_{j\neq i}\gcd(a_j,c_j)\right).
\end{equation}
In particular we have $\Delta_{\bc}(\ba) = \Delta$ when $\bc = \0$.

The following lemma is a slight modification of \cite[Lemma 4.9]{browningbundle}.

\begin{lemma}\label{refined lemma 4.9}
Suppose $A \neq \square$. Let $\bc \in \Z^4$ be such that $F^{*}_{\ba}(\bc)=0$. Then 
$$\Sigma(x;\bc) \ll A^{3/16+\eps}\Delta_{\bc}(\ba)^{3/8}x^{1/2+\eps}.$$
\end{lemma}

\begin{proof}
We use the fact that $S_{q,\ba}(\bc)$ is multiplicative in $q$, as proved by Heath-Brown in the discussion following \cite[Lemma 28]{MR1421949}. This allows us to write
\begin{equation}\label{using multiplicativity}
\Sigma(x;\bc) = \sum_{\substack{q_2\leq x\\ q_2|(2A)^{\infty}}}q_2^{-3}S_{q_2,\ba}(\bc)\sum_{\substack{q_1 \leq x/q_2\\ \gcd(q_1,2A) = 1}}q_1^{-3}S_{q_1,\ba}(\bc),    
\end{equation} 
where the notation $q_2|(2A)^{\infty}$ means that every prime dividing $q_2$ also divides $2A$. By \cite[Lemma 4.6]{browningbundle}, we have
$$\sum_{\substack{q_1 \leq x/q_2\\ \gcd(q_1,2A) = 1}}q_1^{-3}S_{q_1}(\bc) = \sum_{\substack{q_1 \leq x/q_2 \\  \gcd(q_1,2A) = 1}}\left(\frac{A}{q_1}\right)\frac{\phi(q_1)}{q_1},$$ 
where $\phi$ denotes the Euler totient function. Applying the Burgess bound as in the proof of \cite[Lemma 4.9]{browningbundle}, we obtain
$$\sum_{\substack{q_1 \leq x/q_2\\ \gcd(q_1,2A) = 1}}q_1^{-3}S_{q_1,\ba}(\bc)\ll \frac{|A|^{3/16+\eps}x^{1/2}}{q_2^{1/2}}.$$
Returning to (\ref{using multiplicativity}), we have 
\begin{equation}\label{combine with this}
\Sigma(x;\bc)\ll |A|^{3/16+\eps}x^{1/2}\sum_{\substack{q_2\leq x\\ q_2|(2A)^{\infty}}}\frac{|S_{q_2,\ba}(\bc)|}{q_2^{7/2}}.
\end{equation}

Now \cite[Lemma 4.5]{browningbundle} states that $S_{q,\ba}(\bc) \ll q^3\prod_{i=1}^4 \gcd(q,a_i,c_i)$. Moreover, there are $O(x^{\eps}A^{\eps})$ choices for $q_2 \leq x$ with $q_2|(2A)^{\infty}$. Therefore 

\begin{equation}\label{a sup}
\sum_{\substack{q_2\leq x\\ q_2|(2A)^{\infty}}}\frac{|S_{q_2,\ba}(\bc)|}{q_2^{7/2}}\ll x^{\eps}A^{\eps}\sup_{\substack{q|(2A)^{\infty}}}\left(\frac{\prod_{i=1}^4 \gcd(q,a_i,c_i)^{1/2}}{q^{1/2}}\right).
\end{equation}
Let $p$ be a prime. We define
$$L_p = \nu_{p}\left(\frac{\prod_{i=1}^4 \gcd(q,a_i,c_i)}{q}\right),\quad k_p = \nu_p(q).$$ 
For an index $i=1,\ldots, 4$, we let $m_{i,p}$ denote the $i$th smallest element of $\nu_p(\gcd(a_1,c_1)),\ldots, \nu_p(\gcd(a_4,c_4))$.
Then 
$$L_p = \sum_{i=1}^4 \left(\min(k_p,m_{i,p}) - k_p\right).$$

 It can be checked that the maximum possible value of $L_p$ is attained by choosing $k_p = m_{3,p}$, and so $L_p \leq m_{1,p}+ m_{2,p}+m_{3,p}$. We have $L_p\leq 0$ unless $p|\Delta_{\bc}(\ba)$. Hence the supremum in (\ref{a sup}) is bounded by 
 $$\prod_{p|\Delta_{\bc}(\ba)}p^{(m_{1,p}+m_{2,p}+m_{3,p})/2}.$$
 For any prime $p$, we have 
 $$\nu_p(\Delta_{\bc}(\ba)) = m_{1,p}+m_{2,p}+m_{3,p}+\min(m_{1,p}+m_{2,p}+m_{3,p},m_{4,p}) \geq \frac{4}{3}(m_{2,p}+m_{3,p}),$$ 
and hence the supremum in (\ref{a sup}) is bounded by $\Delta_{\bc}(\ba)^{1/3}$. Recalling (\ref{combine with this}) we obtain the desired estimate for $\Sigma(x;\bc)$. 
\end{proof}

\begin{proposition}\label{asymptotics for main term} Suppose $A \neq \square$, and let $M(B)$ be as defined in (\ref{definition of M(B)}). Then
$$M(B) = \frac{\mathfrak{G}_{\ba}\sigma_{\infty}(w)B}{|A|^{1/2}} + O\left(\frac{\eta^{-2}B^{3/4+\epsilon}\Delta^{3/8}}{|A|^{5/16}}\right).$$
\end{proposition}

\begin{proof} We recall the estimate for $M(B)$ from (\ref{split the main term up}) and the definition of $T(R)$ from \eqref{TR}. We may restrict the sum to the range $R<q\ll Q$, since $h(r,G(\bx))=0$ unless $q \ll Q$. Using Lemma \ref{refined lemma 4.9} with $x\ll Q$ and $\bc =\bm{0}$, together with  \cref{lemma 4.3*}, and partial summation, we obtain 
\begin{align*}
    T(R) &=\frac{-I_R(\0)}{Q^2R}\Sigma(R;\0)-\frac{1}{Q^2}\int_{R}^{Q}\Sigma(x;\0)\frac{\del}{\del x}\left(\frac{I_{x,\ba}(\0)}{x}\right)\mathrm{d}x\\
    &\ll\frac{P_1\cdots P_4}{BR}\sup_{R\leq t \ll Q}|\Sigma(t;\0)|\\
    &\ll \frac{P_1\cdots P_4}{B R}|A|^{3/16+\eps}\Delta^{3/8}B^{1/4+\epsilon}\\
    &\ll \frac{\eta^{-2}B^{3/4+\eps}\Delta^{3/8}}{|A|^{5/16}},
\end{align*}
which is the error term claimed in the proposition. Finally, the same error term is also obtained when we extend the sum $\sum_{q \leq R}q^{-4}S_{q,\ba}(\0)$ in (\ref{split the main term up}) to the singular series $\mathfrak{G}_{\ba} = \sum_{q=1}^{\infty}q^{-4}S_{q,\ba}(\0)$, as can be seen by applying Lemma \ref{refined lemma 4.9} and partial summation in a very similar manner to above.
\end{proof}

\subsection{The error term}

We now study the error term coming from the case $\bc \neq \0$. We begin with a lemma which is similar to \cite[Lemma 4.2 (i)]{browningbundle}.

\begin{lemma}\label{modified lemma 4.2} For any nonzero $\bc \in \Z^4$ and any integer $N\geq 0$, we have
$$I_{q,\ba}(\bc) \ll_N \frac{P_1 \cdots P_4}{r}\left(1 + \frac{r}{\eta}\right)^N \min_{1\leq i\leq 4}\left(\frac{|a_i|^{1/2}}{|c_i|}\right)^N.$$
\end{lemma}

\begin{proof} We apply integration by parts $N$ times to the integral in (\ref{iqa0}), where we differentiate $f(\bx):=w(\bx)h(r,G(\bx))$ and integrate $g(\bx):=e(\bv\cdot\bx)$. Each integral of $g(\bx)$ with respect to $x_i$ introduces a factor of $(2\pi i q^{-1}P_ic_i)^{-1}$ .

We recall from (\ref{derivatives of smooth weights}) that for any $i\in \{1,\ldots, 4\}$ and any $N\geq 1$, 
$$\left|\frac{\del^{N}}{\del x_i^{N}}w(\bx)\right| \ll_N \eta^{-N}.$$
Using the product rule, we have
$$\left|\frac{\del^N}{\del x_i^{N}}f(\bx)\right|\ll_{N} \max_{0\leq j \leq N}\left|\eta^{j-N}\frac{\del^{j}}{\del x_i^{j}}h(r,G(\bx))\right|.$$ 
Clearly
$$\left|\frac{\del^{j}}{\del x_i^{j}}G(\bx)\right|\ll 1$$ 
for all $\bx\in \op{supp}w$. By \cite[Equation (4.10)]{browningbundle}, we have  
\begin{equation}\label{derivatives of h}\frac{\del^{j}h(r,y)}{\del y^j}\ll_j r^{-1-j}.\end{equation}
Hence by the chain rule, we have
\begin{equation}\left|\frac{\del^{N}}{\del x_i^{N}}f(\bx)\right|\ll_N \max_{0\leq j\leq N}r^{-1-j}\eta^{j-N} = r^{-1}\min(r,\eta)^{-N}\leq r^{-1}\left(\frac{1}{r}+ \frac{1}{\eta}\right)^N.\label{derivatives of f}\end{equation}

After performing integration by parts $N$ times, we apply (\ref{derivatives of f}) to all derivatives of $f(\bx)$ that appear, and the trivial estimate $|e(-\bv\cdot\bx)|\leq 1$, to obtain for any $i\in \{1,\ldots, 4\}$,
\begin{align*}
I_{q,\ba}(\bc) &\ll_N \frac{P_1\cdots P_4}{r(q^{-1}|P_ic_i|)^N}\left(\frac{1}{r}+\frac{1}{\eta}\right)^N.
\end{align*}
Recalling that $P_i = (B/|a_i|)^{1/2}, Q=B^{1/2}, r = q/Q$, and choosing the index $i$ such that $|P_ic_i|$,  is maximised, the last expression rearranges to give the desired estimate.
\end{proof}

\begin{remark}\label{remarkably} Lemma \ref{modified lemma 4.2} allows us to assume $|c_i| \ll \eta^{-1}|a_i|^{1/2}B^{\eps}=C_i$ for all $i\in\{1,\ldots, 4\}$, since outside this range the estimate in Lemma \ref{modified lemma 4.2} can be made smaller than any power of $B$ by an appropriate choice of $N$.
\end{remark}

The following lemma is a variant of the first derivative test and is based on \cite[Lemma 10]{MR1421949}.

\begin{lemma}\label{first derivative test} Let $f(\bx) = \theta G(\bx) - \bv\cdot \bx$, where $\bv \in \R^4$ and $\theta \in \R$ are such that $|\bv|\geq 5|\theta|$. Then for any integer $N\geq 0$ and any $w$ in the class of smooth weight functions $S$ from Definition \ref{class of smooth weights S}, we have
$$\int_{\R^4}w(\bx)e(f(\bx))\mathrm{d}\bx \ll_N (\eta|\bv|)^{-N}.$$
\end{lemma}

\begin{proof} The case $N=0$ is trivial. We define $f_j(\bx)=\del f(\bx)/\del x_j = \pm 2\theta x_j -v_i$. By the assumption $|\bv|\geq 5|\theta|$, there is some index $j \in \{1,\ldots, 4\}$ such that $|f_j(\bx)|\gg |\bv|$ for any $\bx \in [-2,2]^4$. Without loss of generality, we may assume $j=1$.

We write 
$$w(\bx)e(f(\bx)) = \frac{w(\bx)}{2\pi i f_1(\bx)}\frac{\del}{\del x_1}e(f(\bx))$$
and integrate by parts with respect to $x_1$ to obtain
$$\int_{\R^4}w(\bx)e(f(\bx))\mathrm{d}\bx = -(2\pi i\eta|\bv|)^{-1}\int_{\R^4}e(f(\bx))\wt\mathrm{d}\bx,$$
where
$$\wt = \frac{\del}{\del x_1}\left(\frac{\eta |\bv|w(\bx)}{f_1(\bx)} \right) = \frac{\del}{\del x_1}(\eta w(\bx))\frac{|\bv|}{f_1(\bx)}+ \eta w(\bx)\frac{\del}{\del x_1}\left(\frac{|\bv|}{f_1(\bx)}\right).$$

It remains to show that $\widetilde{w}$ belongs to the class of smooth weight functions $S$ from Definition \ref{class of smooth weights S}, since the result then follows by induction on $N$. But this is immediate from the observations that for any $\bx \in [-2,2]^4,$ 
\begin{equation*}\frac{|\bv|}{f_1(\bx)}\ll 1, \quad \frac{\del}{\del x_1}\left(\frac{|\bv|}{f_1(\bx)}\right)\ll 1, \quad \textrm{ and } \quad \frac{\del}{\del x_1}(\eta w) \in S. \qedhere \end{equation*}
\end{proof}

We require one further estimate for $I_{q,\ba}(\bc)$, which involves the second derivative test.
\begin{lemma}\label{modified version of lemma 4.2 ii)} Let $\bc \neq \0$. Then

$$I_{q,\ba}(\bc) \ll \frac{\eta^{-4}B^{3/2+\epsilon}q}{|A|^{1/2}}\min_{1\leq i\leq 4}\left(\frac{|a_i|^{1/2}}{|c_i|}\right).$$
\end{lemma}

\begin{proof} We would like to apply Fourier inversion to $w(\bx)h(r,G(\bx))e(-\bv\cdot \bx)$, the integrand appearing in the definition of $I_{q,\ba}(\bc)$. The function $\bx \mapsto h(r,G(\bx))$ does not have compact support. We define the smooth weight $w_0:\R\ra \R_{\geq 0}$ by $w_0(x) = \psi(x/17)$, where $\psi(x)$ is the standard bump function in one variable. Then we define the smooth weight $w_3:R^4 \ra \R_{\geq 0}$ by 
 $$w_3(\bx) = \begin{cases} \frac{w(\bx)}{w_0(G(\bx))}, &\mbox{if } w_0(G(\bx)) \neq 0,\\
0, & \mbox{otherwise.}
\end{cases}$$
Note that $w$ is supported on $[-2,2]^4$, and in this region $G(\bx)$ takes values in $[-16,16]$. Therefore $w_0(G(\bx)) \gg 1$ for all $\bx \in \mathrm{supp}(w)$. We deduce that $\mathrm{supp}(w_3) = \mathrm{supp}(w)$ and $w(\bx) = w_3(\bx)w_0(G(\bx))$ for all $\bx \in \R^4$. Moreover, since all the derivatives of $w_0(G(\bx))$ are $O(1)$ for any $\bx \in \supp(w)$, we have $w_3 \in S$. Applying Fourier inversion, we obtain 
\begin{equation}\label{fourier inversion} 
I_{q,\ba}(\bc) = P_1 \cdots P_4 \int_{-\infty}^{\infty} p(\theta) \int_{\R^4} w_3(\bx) e(\theta G(\bx) - \bv \cdot \bx) \,\mathrm{d}\bx\,\mathrm{d}\theta,
\end{equation}
where 
$$ p(\theta) = \int_{-\infty}^{\infty} w_0(k)h(r,k)e(-\theta k)\,\mathrm{d}k.$$
 
We have the estimate $p(\theta) \ll 1$. Indeed, taking $m=n=0$ and $N=2$ in \cite[lemma 5]{MR1421949} we see that 
$$h(r,k) \ll r^{-1}(r^2 + \min(1,(r/|k|)^2),$$
and so
\begin{align*}p(\theta)\ll r^{-1}\left(\int_{|k|<r}1 \mathrm{d}k + \int_{r\leq|k|\leq 17}\left(\frac{r}{|k|}\right)^2\mathrm{d}k \right)\ll r^{-1}(r + r^2\cdot r^{-1})\ll 1.
\end{align*}

To deal with the inner integral in (\ref{fourier inversion}), which we denote by $I(\theta;\bv)$, we divide into the cases $5|\theta| \leq |\bv|$ and $5|\theta|\geq |\bv|$. In the former case, we apply Lemma \ref{first derivative test} with $N=2$ to obtain $I(\theta;\bv)\ll (\eta|\bv|)^{-2}$. The contribution to $I_{q,\ba}(\bc)$ is 
$$ (\eta|\bv|)^{-2}P_1\cdots P_4 \int_{5|\theta|\leq |\bv|} 1 \mathrm{d}\theta \ll \eta^{-2}|\bv|^{-1}P_1\cdots P_4.$$

In the case $5|\theta|\geq |\bv|$, we use the arguments from \cite[Lemma 3.2]{MR3652248} to obtain
\begin{equation}\label{1D integrals}
I(\theta;\bv)\ll \left\{\int_{\R^4}|\hat{w_3}(\by)| \mathrm{d}\by\right\} \sup_{\by \in \R^4}\left|\prod_{i=1}^4\int_{[-2,2]} e(\theta\eps_i x_i^2 +x_i(y_i \pm v_i))\mathrm{d}x_i \right|,
\end{equation}
where $\hat{w_3}$ denotes the Fourier transform of $w_3$. The first factor on the right hand side of (\ref{1D integrals}) is the $L^1$-norm of $\hat{w_3}$, which we denote by $\|\hat{w_3}\|_{L^1}$. The $v_i$'s can be absorbed into the supremum over $\by$, and so it remains to study 
$$\int_{[-2,2]}e(\pm \theta x_i^2 - y_ix_i)\mathrm{d}x_i.$$
We can write the integrand as $e(\pm\theta \Phi(x_i))$, where $\Phi(x_i) = x_i^2 + x_iy_i\theta^{-1}$. Then $|\Phi''(x_i)|\geq 1$ throughout the interval $[-2,2]$, and so we can apply the second derivative test as found in \cite[Ch. 8, Proposition 2.3]{MR2827930} to bound this integral by $|\theta|^{-1/2}$. Therefore
$$I(\bv;\theta) \ll \|\hat{w_3}\|_{L_1}|\theta|^{-2},$$
Returning to (\ref{fourier inversion}), the contribution to $I_{q,\ba}(\bc)$ from the range $5|\theta| \geq |\bv|$ is bounded by
$$\|\hat{w_3}\|_{L_1}P_1\cdots P_4 \int_{|\theta|\gg |\bv|}|\theta|^{-2} \mathrm{d}\theta\ll \|\hat{w_3}\|_{L_1}P_1\cdots P_4 |\bv|^{-1}.$$

Since 
$$\frac{P_1\cdots P_4}{|\bv|} = \frac{B^{3/2}q}{|A|^{1/2}}\min_{1\leq i\leq 4}\left(\frac{|a_i|^{1/2}}{|c_i|}\right),$$
it remains only to show that $\|\hat{w}_3\|_{L^1} \ll \eta^{-4}$. We have 
$$\|\hat{w}_3\|_{L^1} = \int_{|\bx| \leq \eta^{-1}}|\hat{w_3}(\by)|\,\mathrm{d}\by + \int_{|\bx| \geq \eta^{-1}}|\hat{w_3}(\by)|\,\mathrm{d}\by.$$
The first integral is trivially $O(\eta^{-4})$. For the second integral, we have 
$$ \int_{|\bx|\geq \eta^{-1}}|\hat{w_3}(\by)|\,\mathrm{d}\by \ll  \int_{\substack{|\bx|\geq \eta^{-1}\\ |y_1| = \max_i|y_i|}}|\hat{w_3}(\by)|\,\mathrm{d}\by.$$
We can now apply integration by parts five times with respect to $x_1$ to obtain
\begin{align*}
    \int_{|\by|\geq \eta^{-1}}|\hat{w_3}(\by)|\,\mathrm{d}\by &\ll \int_{\substack{|\by|\geq \eta^{-1}\\|y_1|=\max_i|y_i|}}\left|\int_{\R^4}\frac{\del^5w_3(\bx)}{\del x_1^5}\frac{e(-\bx \cdot \by)}{(-2\pi i y_1)^5}\mathrm{d}\bx\right|\mathrm{d}\by\\
    &\ll \eta^{-5}\int_{
    \substack{|y_1|\geq \eta^{-1}\\
    |y_2|,|y_3|,|y_4|\leq |y_1|}}y_1^{-5}\mathrm{d}\by\\
    &\ll \eta^{-5}\int_{y_1\geq \eta^{-1}}y_1^{-2}\,\mathrm{d}y_1\\
    &\ll \eta^{-4}.\qedhere
\end{align*}
\end{proof}

We recall the notation $C_i=\eta^{-1}B^{\eps}|a_i|^{1/2}$ from Remark \ref{remarkably}. We are now ready to estimate the quantity 
\begin{equation}\label{eab}
E_{\ba}(B) = \frac{1}{Q^2}\sum_{\substack{\bc \in \Z^4\bs\{\0\}\\ |c_i|\ll C_i}}\sum_{q=1}^{\infty}q^{-4}S_{q,\ba}(\bc)I_{q,\ba}(\bc).
\end{equation}
Note that $I_{q,\ba}(\bc)$ vanishes for $q\gg Q$ by the properties of the function $h$, and so we may restrict the $q$-sum to a sum over $q\ll Q$. We recall the notation $\Sigma(x;\bc) = \sum_{q \leq x}S_{q,\ba}(\bc)$. We would like to use \cref{modified version of lemma 4.2 ii)} together with the bound for $\Sigma(x;\bc)$ from \cite[Lemma 4.7]{browningbundle} to estimate the inner sum of (\ref{eab}). Unfortunately, \cite[Lemma 4.7]{browningbundle} requires the dual form $F_{\ba}^*(\bc)$ not to vanish. This was always true in \cite[Section 4]{browningbundle} due to the assumptions made in the setup, but we must treat this case separately. To this end, we let $E_{\ba}(B) = E_{\ba,1}(B) + E_{\ba,2}(B)$, where in $E_{\ba,1}(B)$ we add the restriction $F_{\ba}^*(\bc) \neq 0$ to the $\bc$-sum and in $E_{\ba,2}(B)$ we sum over the remaining values of $\bc$ where $F_{\ba}^*(\bc) = 0$. 

\subsubsection{Analysis of $E_{\ba,1}(B)$}

Using Lemma \ref{modified version of lemma 4.2 ii)}, we have
$$E_{\ba,1}(B) \ll \frac{\eta^{-4}B^{1/2+\eps}}{|A|^{1/2}}\sum_{\substack{\bc \in \Z^4\bs\{\0\}\\ |c_i|\ll C_i}} \min_{1\leq i\leq 4}\left(\frac{|a_i|^{1/2}}{|c_i|}\right)\sum_{q\ll Q}q^{-3}S_{q,\ba}(\bc).$$
For a fixed choice of $\bc$, let $j(\bc)\in\{1,\ldots,4\}$ denote the index where $|a_i|^{1/2}|c_i|^{-1}$ is minimized. Note in particular that $c_{j(\bc)}\neq 0$. We have
$$E_{\ba,1}(B) \ll \frac{\eta^{-4}B^{1/2+\eps}}{|A|^{1/2}}\sum_{j=1}^4|a_j|^{1/2}\sum_{\substack{\bc \in \Z^4\bs\{\0\}\\ |c_i|\ll C_i\\j(\bc) = j}} \frac{1}{|c_j|}\sum_{q\ll Q}q^{-3}S_{q,\ba}(\bc).$$
An application of \cite[Lemma 4.7]{browningbundle} with $x = Q$ yields
$$\sum_{\substack{\bc \in \Z^4\bs\{\0\}\\ |c_i|\ll C_i\\j(\bc) = j}} \frac{1}{|c_j|}\sum_{q\ll Q}q^{-3}S_{q,\ba}(\bc)\ll  B^{\eps}\sum_{\substack{\bc \in \Z^4\bs\{\0\}\\ |c_i|\ll C_i}}\frac{1}{|c_j|}\prod_{i=1}^4\gcd(a_i, c_i)^{1/2}.$$
For $i\neq j$, we have 
$$\sum_{0\leq|c_i| \ll C_i}\gcd(a_i, c_i)^{1/2} \ll C_iB^{\eps},$$
and for $i=j$, we have 
$$\sum_{0<|c_j| \ll C_j}\frac{\gcd(a_j, c_j)^{1/2}}{|c_j|} \ll B^{\eps}.$$
Hence
$$E_{\ba,1}(B) \ll \frac{\eta^{-4}B^{1/2+\eps}}{|A|^{1/2}}\sum_{j=1}^4 |a_j|^{1/2}\prod_{i\neq j}C_i\ll \eta^{-7}B^{1/2+\eps}.$$

\subsubsection{Analysis of $E_{\ba,2}(B)$}\label{E2P}

We write $E_{\ba,2}(B) = E_{\ba,2}^{(1)}(B) + E_{\ba,2}^{(2)}(B)$, where
\begin{align*}
E_{\ba,2}^{(1)}(B)  &= \frac{1}{B}\sum_{\substack{\bc \in \Z^4\bs\{\0\}\\ |c_i| \ll C_i\\F_{\ba}^*(\bc)= 0}}\sum_{q\leq M}q^{-4}S_{q,\ba}(\bc)I_{q,\ba}(\bc),\\
E_{\ba,2}^{(2)}(B)  &= \frac{1}{B}\sum_{\substack{\bc \in \Z^4\bs\{\0\}\\ |c_i| \ll C_i\\F_{\ba}^*(\bc)= 0}}\sum_{M\leq q\ll Q}q^{-4}S_{q,\ba}(\bc)I_{q,\ba}(\bc),
\end{align*}
For a real parameter $M$ to be determined later. To bound $E_{\ba,2}^{(1)}(B)$, we apply \cref{modified version of lemma 4.2 ii)} to obtain
$$E_{2,a}(P) \ll \frac{\eta^{-4}B^{1/2+\eps}}{|A|^{1/2}}\sum_{j=1}^4\sum_{\substack{\bc \in \Z^4\bs\{\0\}\\  |c_i|\ll C_i\\j(\bc) = j}}    \frac{|a_j|^{1/2}}{|c_j|}  \sum_{q\leq M}q^{-3}|S_{q,\ba}(\bc)|.$$
By \cite[Lemma 4.5]{browningbundle}, we have $|S_{q,\ba}(\bc)| \ll q^3\prod_{i=1}^4 \gcd(a_i, c_i)^{1/2}$, and hence summing trivially over $q$ and proceeding as before,  we obtain 
$$E_{\ba,2}^{(1)}(B) \ll \eta^{-7}B^{1/2+\eps}M.$$

To bound $E_{\ba,2}^{(2)}(B)$ we require some cancellation from the sum over $q$. To achieve this, we use summation by parts, as in the treatment of the main term in Proposition \ref{asymptotics for main term}, and then apply Lemma \ref{refined lemma 4.9} to obtain a better estimate for the exponential sums. We also perform the $q$-sum over $R\leq q\leq 2R$ and later take a dyadic sum over $M\leq R \ll Q$. Summation by parts yields
\begin{equation}
 \sum_{R\leq q \leq 2R}q^{-4}S_{q,\ba}(\bc)I_{q,\ba}(\bc) \ll -\Sigma(R;\bc)\frac{I_{R,\ba}(\bc)}{R}-\int_{R}^{2R} \Sigma(x;\bc)\frac{\del}{\del x}\left(\frac{I_{x,\ba}(\bc)}{x}\right) \,\mathrm{d}x.\label{partial summation for E2P}   
\end{equation}
Therefore, from Lemma \ref{lemma 4.3*}, we obtain
\begin{equation}\label{5/8}
 \frac{1}{B}\sum_{R\leq q \leq 2R}q^{-4}S_{q,\ba}(\bc)I_{q,\ba}(\bc) \ll \frac{\eta^{-1}P_1\cdots P_4}{BR}\sup_{R \leq x \leq 2R}|\Sigma(x;\bc)|.
\end{equation}
We recall the definition of $\Delta_{\bc}(\ba)$ from (\ref{sneaky delta def}). Using the estimate for $|\Sigma(x;\bc)|$ from Lemma \ref{refined lemma 4.9}, we have
\begin{align}
    E_{\ba,2}^{(2)}(B)&\ll \frac{\eta^{-1}P_1\cdots P_4|A|^{3/16}}{B}\sum_{\substack{R \textrm{ dyadic}\\M\leq R \ll Q}}R^{-1/2+\eps}\sum_{\substack{\bc \in \Z^4\bs\{\0\}\\ |c_i| \ll C_i\\F_{\ba}^*(\bc)= 0}}\Delta_{\bc}(\ba)^{3/8}\nonumber\\
    &\ll \frac{\eta^{-1}B^{1+\eps}}{M^{1/2}|A|^{5/16}}\sum_{\substack{\bc \in \Z^4\bs\{\0\}\\ |c_i| \ll C_i\\F_{\ba}^*(\bc)= 0}}\Delta_{\bc}(\ba)^{3/8}.\label{hard case}
\end{align}

It is possible to obtain the asymptotic formula from Theorem \ref{the main theorem} (albeit with a smaller power saving) by bounding the sum in (\ref{hard case}) trivially by $C_1\cdots C_4 \Delta^{3/8}$. However, in order to obtain the error terms claimed in Theorem \ref{theorem 1.4 again}, we now make use of the constraint $F^*_{\ba}(\bc)=0$ in the $\bc$-sum to obtain a more refined estimate. In the following two lemmas, we adopt the notation that $\bv\bw = (v_1w_1,\ldots, v_4w_4)$ for vectors $\bv,\bw \in \Z^4$. 

\begin{lemma}\label{basic property of delta}
For any $\bm{d},\bm{e} \in (\Z_{\neq 0})^4$, we have
$$\Delta(\bm{d}\bm{e}) \leq (d_1\cdots d_4)^2\Delta(\bm{e}). $$
\end{lemma}
\begin{proof}
Fix a prime $p$, and define $\delta_i = \nu_p(d_i), \varepsilon_i = \nu_p(e_i)$ for $i\in \{1,\ldots, 4\}$. Without loss of generality, we may assume that $\varepsilon_1\leq \cdots \leq \varepsilon_4$. Then
\begin{align*}
\nu_p(\Delta(\bm{d}\bm{e}))&= \sum_{i=1}^4 \min\left(\delta_i+\varepsilon_i,\sum_{j \neq i}(\delta_j+\varepsilon_j)\right)\\
&\leq \sum_{i=1}^3(\delta_i+\varepsilon_i)+\min\left(\delta_4+\varepsilon_4,\sum_{i=1}^3(\delta_i+\varepsilon_i)\right),
\end{align*}
and
$$\nu_p(\Delta(\bm{e})) = \sum_{i=1}^3 \varepsilon_i+\min\left(\varepsilon_4, \sum_{i=1}^3 \varepsilon_i\right).$$
Therefore
\begin{align*}\nu_p(\Delta(\bm{d}\bm{e}))-\nu_p(\Delta(\bm{e})) &\leq 
\begin{cases}
\sum_{i=1}^3 2\delta_i, &\textrm{ if } \varepsilon_1+\varepsilon_2+\varepsilon_3 \leq \varepsilon_4,\\
\sum_{i=1}^4 \delta_i, &\textrm{ otherwise.}
\end{cases}\\
&\leq \nu_p((d_1\cdots d_4)^2).
\end{align*}
Taking a product over all primes $p$ completes the proof of the lemma.
\end{proof}

\begin{lemma}
We have 
$$\sum_{\substack{\bc \in \Z^4\bs\{\0\}\\ |c_i| \ll C_i\\F_{\ba}^*(\bc)= 0}}\Delta_{\bc}(\ba)^{3/8}\ll \eta^{-3}B^{\eps}\Delta^{1/2}.$$

\end{lemma}

\begin{proof}
For a nonzero integer $m$, we define $\sqf(m)$ to be the smallest positive integer $r$ such that $|m|/r$ is a square. For $i\in \{1,\ldots, 4\}$, we decompose $a_i$ into a product $a_i = e_in_ik_i^2$, where 
$$ e_i = \gcd\left(a_i, \prod_{j\neq i}a_j\right),\quad n_i = \sqf(a_i/e_i), \quad k_i^2 = a_i/(e_in_i).$$
By definition, we have $\prod_{i=1}^4 e_i = \Delta$. For any $i\in \{1,\ldots, 4\}$, it is clear that
\begin{equation}\label{dealing with a dual}
    F^*_{\ba}(\bc) = 0 \implies a_i| c_i^2\prod_{j\neq i}a_j  \implies n_ik_i^2|c_i^2 \implies n_ik_i|c_i.
\end{equation}

We divide up the sum over $\bc$ according to how many of the coordinates $c_1,\ldots,c_4$ are equal to zero. At most two of the coordinates can be zero, due to the assumptions $\bc\neq 0$ and $F^*_{\ba}(\bc)=0$. Up to reordering the indices, we obtain
\begin{equation}\label{how many of you are zero}\sum_{\substack{\bc \in \Z^4\bs\{\0\}\\ |c_i| \ll C_i\\F_{\ba}^*(\bc)= 0}}\Delta_{\bc}(\ba)^{3/8}\ll T_0+T_1+T_2,\end{equation}
where
\begin{align*}
    T_0 = \sum_{\substack{\bc \in (\Z_{\neq 0})^4\\ |c_i|\ll C_i\\F_{\ba}^*(\bc)= 0}}\Delta_{\bc}(\ba)^{3/8},
    T_1 =\sum_{\substack{\bc \in (\Z_{\neq 0})^3\\ |c_i| \ll C_i\\F_{\ba}^*((\bc,0))= 0}}\Delta_{(\bc,0)}(\ba)^{3/8},
    T_2=\sum_{\substack{\bc \in (\Z_{\neq 0})^2\\ |c_i| \ll C_i\\F_{\ba}^*((\bc,0,0))= 0}}\Delta_{(\bc,0,0)}(\ba)^{3/8}
\end{align*}
We begin by studying $T_0$. Given that $F^*_{\ba}(\bc)=0$, we only need to take the sum over $c_1,c_2$ and $c_3$, because this implicitly determines (up to sign) the value of $c_4$. Together with (\ref{dealing with a dual}), this implies that
\begin{align}
    T_0 &\leq \sum_{\substack{c_1,c_2,c_3 \in \Z_{\neq 0}\\ |c_i| \ll C_i\\n_ik_i|c_i}}\Delta_{\bc}(\ba)^{3/8}\\
    &\leq \sum_{\substack{c_1,c_2,c_3 \in \Z_{\neq 0}\\ |c_i| \ll C_i/(n_ik_i)}}\prod_{i=1}^4\gcd\left(n_ik_i\gcd(a_i,c_i),\prod_{j\neq i}n_jk_j\gcd(a_j,c_j)\right)^{3/8}\nonumber\\
    &\leq \sum_{\substack{\bm{d}\in (\Z_{\neq 0})^4\\d_i \ll C_i/(n_ik_i), \\d_i|a_i}}\sum_{\substack{c_1,c_2,c_3 \in \Z_{\neq 0}\\ |c_i| \ll C_i/(n_ik_i)\\d_i|c_i}}\Delta(\bm{d}\bm{n}\bk)^{3/8}\nonumber\\
    &\ll \left(\prod_{i=1}^3 \frac{C_i}{n_ik_i}\right)\sum_{\substack{\bm{d}\in (\Z_{\neq 0})^4\\d_i \ll C_i/(n_ik_i)\\d_i|a_i}}\Delta(\bm{n}\bk)^{3/8}\prod_{i=1}^4d_i^{-1/4}\label{the gcd is just 1},
\end{align}
where in the last line we have applied Lemma \ref{basic property of delta}. However, $\Delta(\bm{n}\bk)=1$ by the definition of $e_1,\ldots, e_4$. There are $O(A^{\eps})$ choices for $\bm{d}$ in the sum in (\ref{the gcd is just 1}), and each summand is bounded by $1$. Therefore 
\begin{equation}\label{T0}
T_0 \ll \left(\prod_{i=1}^3 \frac{C_i}{n_ik_i}\right)A^{\eps}.
\end{equation}
Returning to (\ref{T0}) and recalling that $C_i = \eta^{-1}B^{\eps}|a_i|^{1/2}=\eta^{-1}B^{\eps}(e_in_i)^{1/2}k_i$, we conclude that 
\begin{align*}
T_0 &\ll \eta^{-3}B^{\eps}\left(\prod_{i=1}^3\frac{|a_i|^{1/2}}{n_ik_i}\right)\\
&\leq \eta^{-3}B^{\eps}\Delta^{1/2}.
\end{align*}
 
The approach for estimating $T_1$ and $T_2$ is very similar, so we will focus on the main differences. It is sufficient to only use the divisibility conditions from (\ref{dealing with a dual}) in these cases. For $T_1$, all the sums in the above argument will be over vectors indexed by $\{1,2,3\}$, because we have fixed $c_4=0$. We have
\begin{align*}T_1 &\ll \left(\prod_{i=1}^3 \frac{C_i}{n_ik_i}\right)A^{\eps}\Delta(n_1k_1,n_2k_2,n_3k_3,a_4)^{3/8}\\
&= \left(\prod_{i=1}^3 \frac{C_i}{n_ik_i}\right)A^{\eps}\\
&\ll \eta^{-3}B^{\eps}\Delta^{1/2}.
\end{align*}

For $T_2$, we only take sums over vectors indexed by $\{1,2\}$, since we have fixed $c_3=c_4=0$. We have 
\begin{align*}T_2 &\ll \left(\prod_{i=1}^2 \frac{C_i}{n_ik_i}\right)A^{\eps}\Delta(n_1k_1,n_2k_2,a_3,a_4)^{3/8}\\
&= \left(\prod_{i=1}^2 \frac{C_i}{n_ik_i}\right)A^{\eps}(a_3,a_4)^{3/4}\\
&\leq \eta^{-2}B^{\eps}(e_1e_2)^{1/2}(e_3e_4)^{3/8}\\
&\ll \eta^{-2}B^{\eps}\Delta^{1/2}.
\end{align*}
Thus each of $T_0,T_1,T_2$ is bounded by $\eta^{-3}B^{\eps}\Delta^{1/2}$, as required. 
\end{proof}

Recalling (\ref{hard case}), we therefore have

$$E_{2,\ba}^{(2)}(B)\ll \frac{\eta^{-4}B^{1+\eps}\Delta^{1/2}}{M^{1/2}|A|^{5/16}}.$$
To combine with the error term $E_{2,\ba}^{(1)}(B)$, we make the choice 
$$M = \eta B^{1/3}\Delta^{1/3}|A|^{-5/24}.$$
This yields the estimate
\begin{equation}\label{final c nonzero error term}
E_{\ba}(B)\ll \frac{\eta^{-6}B^{5/6+\eps}\Delta^{1/3}}{|A|^{5/24}}+\eta^{-7}B^{1/2+\eps}.\end{equation}
This estimate is larger than the error term from Proposition \ref{asymptotics for main term}.  Combining with Proposition \ref{asymptotics for main term} this completes the proof of Theorem \ref{the main smooth weight result}. 

The final ingredient in the proof of Theorem \ref{theorem 1.4 again} is an estimate for the singular series $\mathfrak{G}_{\ba}$. 

\begin{lemma}\label{the one and only singular series estimate} Let $\ba \in (\Z_{\neq 0})^4$ and assume that $A\neq \square$. Then

$$|\mathfrak{G}_{\ba}|\ll |A|^{\eps}\Delta^{1/4}.$$
\end{lemma}

\begin{proof} Beginning with the definition of the singular series, we have
$$\mathfrak{G}_{\ba}=\sum_{q=1}^{\infty}q^{-4}S_{q,\ba}(\0) = \sum_{q\leq T}q^{-4}S_{q,\ba}(\0) + \sum_{q>T}q^{-4}S_{q,\ba}(\0)$$
for any $T\geq 1$. For the sum over $q>T$, we apply partial summation and Lemma \ref{refined lemma 4.9}. We have
\begin{align*}
\sum_{q>T}q^{-4}S_{q,\ba}(\0) &= \frac{\Sigma(T;\0)}{T}-\int_{T}^{\infty}\Sigma(x;\0)\frac{\del}{\del x}(x^{-1})\,\mathrm{d}x\nonumber\\
&\ll T^{-1/2 + \eps}|A|^{3/16+\eps}\Delta^{3/8}.\nonumber
\end{align*}
By choosing $T=|A|^2$, we can ensure that the contribution from this region is negligible.

For the remaining sum over $q\leq T$, we follow a similar argument to the proof of Lemma \ref{refined lemma 4.9}. Using the multiplicativity of $S_{q,\ba}(\0)$ in $q$, we have
\begin{align*}\sum_{q \leq T}q^{-4}S_{q,\ba}(\0) &\ll \left( \sum_{\substack{q_1 \leq T\\\gcd(q_1,2A)=1}}\left(\frac{A}{q_1}\right)\frac{\phi(q_1)}{q_1^2}\right)\left( \sum_{\substack{q_2\leq T/q_1 \\ q_2|(2A)^{\infty}}}q_2^{-4}S_{q_2,\ba}(\0)\right)\\
&\ll T^{\eps}|A|^{\eps}\max_{q|(2A)^{\infty}}\left(\frac{\prod_{i=1}^4\gcd(q,a_i)^{1/2}}{q}\right),\end{align*}
where in the second line we have estimated the sum over $q_1$ trivially, and the sum over $q_2$ by applying \cite[Lemma 4.5]{browningbundle}. 
We let $m_{i,p}$ denote the $i$th smallest element of $\nu_p(a_1),\ldots, \nu_p(a_4)$. We define
$$L_p = \nu_{p}\left(\frac{\prod_{i=1}^4\gcd(q,a_i)}{q^2}\right),\quad k_p = \nu_p(q).$$ 
Then 
$$L_p = \sum_{j=1}^4 \min(k_p,m_{i,p}) - 2k_p.$$
The maximum possible value of $L_p$ is attained by choosing $k_p = m_{2,p}$, and so $L_p \leq m_{1,p}+ m_{2,p}$. We have
$$\nu_p(\Delta) = m_{1,p}+m_{2,p}+m_{3,p}+\min(m_{1,p}+m_{2,p}+m_{3,p},m_{4,p}) \geq 2(m_{1,p}+m_{2,p}),$$
and so $L_p \leq \nu_p(\Delta)/2$. Taking a product over $p|2A$ completes the proof of the lemma.
\end{proof}

\section{Proof of Theorem \ref{the main theorem}}\label{section proof of theorem 1.2}

We recall that the quantity $N_{\ba}(B)$ studied in Section \ref{section the circle method} did not involve any primitivity conditions. In order to insert the condition $\gcd(z_1, \ldots,z_4)=1$, we apply an inclusion-exclusion argument which is a special case of \cite[Section 3]{browning2019arithmetic}.

 We begin by fixing some notation which we will use throughout this section. Let $z_1,\ldots, z_4$ denote nonzero squareful numbers, and $x_1,\ldots,x_4 \in \N$, $y_1, \ldots, y_4 \in \Z_{\neq 0}$ the unique integers such that $z_i = x_i^2y_i^3$ and $y_i$ is square-free for all $i$. Let $A = a_1\cdots a_4, R=r_1\cdots r_4, S=s_1\cdots s_4$ and $Y = y_1\cdots y_4$. For vectors $\bv, \bw \in \N^4$, we write $\bv|\bw$ to mean $v_i|w_i$ for all $i=1,\ldots, 4$. For an integer $m$, we write $\bv|m$ for $v_i|m$ for all $i=1,\ldots, 4$, and $\gcd(m,\bv)$ for $(\gcd(m,v_1),\ldots, \gcd(m,v_4))$. We also define $\bv^{[m]}$ to be the vector in $\N^4$ with $i$th coordinate $v_i^{[m]} = \prod_{p|m}p^{\nu_p(v_i)}$, where $\nu_p$ denotes the $p$-adic valuation.

We recall from Proposition \ref{large coeffs are o(B)} that 
\begin{equation}\label{large error}
N(B) = N(D,B) + O(B^{1+\eps}D^{-1/12}),
\end{equation}
where $N(B)$ is defined in (\ref{main counting problem}) and $N(D,B)$ is defined in the same way but with the additional constraint $|Y|\leq D$. For $\br, \bms\in \N^4$ and $s_0\in \N$, we define
\begin{align*}
\mathcal{N}(B;\mathbf{r},\mathbf{s},s_0) =\left\{\mathbf{z} \in (\Z_{\neq 0})^4:
\begin{tabular}{l}
$|\mathbf{z}|\leq B, |Y|\leq D, Y \neq \square$,\\  
$z_1 + \cdots + z_4 = 0, \br|\by, \bms|\bx, s_0|\bx$
\end{tabular}
\right\}.
\end{align*}

\definition\label{omega} Given $\mathbf{r},\mathbf{s}\in \N^4$ and $s_0 \in \N$, we define $\omega(\br,\bms,s_0)$ as follows:

\begin{enumerate}
    \item $\omega(\br, \bms,s_0) = \mu(s_0)\prod_{p}\omega(\br^{[p]},\bms^{[p]},1)$. In particular, $\omega(\mathbf{1}, \mathbf{1},s_0) = \mu(s_0).$
    \item If one of $r_1, \ldots, r_4, s_1, \ldots, s_4$ is not square-free, then $\omega(\br, \bms,s_0) = 0.$
    \item If $\gcd(s_1, \ldots, s_4)>1$ then $\omega(\br,\bms,s_0)=0$.
    \item If $\gcd(s_0,\bms) \neq \1$, then $\omega(\br,\bms,s_0)=0$.
    \item If $p|RS$ but $p\nmid r_is_i$ for some $i$, then $\omega(\br^{[p]},\bms^{[p]},1)=0$.
    \item If $p|r_is_i$ for every $i$, $\gcd(s_1, \ldots, s_4)=1$, $\gcd(s_0,\bms)=1$, and $r_i,s_i$ are square-free for every $i$, then $k := \nu_p(RS)$ satisfies $4\leq k\leq 7$. Define $\omega(\br^{[p]}, \bms^{[p]},1) = (-1)^{k+1}$. 
     
\end{enumerate}

The motivation for this choice of $\omega$ comes from the following lemma.

\begin{lemma}\label{inclusion exclusion} We have

\begin{equation}\label{Mobius inversion statement}N(D,B) = \sum_{\substack{\br, \bms \in \N^4\\ s_0 \in \N}}\omega(\br, \bms, s_0)\#\mathcal{N}(B;\br, \bms,s_0).\end{equation}
\end{lemma}

\begin{proof} We write the right hand side of (\ref{Mobius inversion statement}) as

\begin{equation}\label{Mobius inversion}\sum_{\bz \in \mathcal{N}(B;\mathbf{1}, \mathbf{1},1)} \sum_{\substack{\br, \bms,s_0\\ \br|\by, \bms|\bx, s_0|\bx}}\omega(\br, \bms,s_0).\end{equation}

We would like to show that the inner sum is the indicator function for the condition $\gcd(z_1,\ldots, z_4)=1$. If $\gcd(z_1, \ldots, z_4) = 1$, then by property $(5)$ of Definition \ref{omega}, the only nonzero term in the inner sum of (\ref{Mobius inversion}) is the term $\omega(\mathbf{1},\mathbf{1},1) = 1$. From now on, we suppose that $\gcd(z_1,\ldots,z_4)>1$. By properties $(1)$ and $(2)$ from Definition \ref{omega}, it suffices to show that for any prime $p$ dividing $\gcd(z_1,\ldots,z_4)$, we have 
\begin{equation}\label{omega at primes}
\sum_{\substack{\br, \bms \in \N^4, s_0\in \N\\ \br|(p,\by), \bms|(p,\bx), s_0|(p,\bx)}}\omega(\br, \bms,s_0)=0.
\end{equation}
The condition $s_0|(p,\bx)$ implies that $s_0=1$ or $s_0=p$. If $s_0=p$ and $\omega(\br,\bms,s_0)\neq 0$, then by $(4)$ and $(5)$ of Definition \ref{omega}, we have $\bms=\1$ and $\br=\1$ or $\br = (p,p,p,p)$. Therefore the left hand side of (\ref{omega at primes}) becomes
$$ \mu(p)(\omega(\1,\1,1)+\omega((p,p,p,p),\1,1)) = -(1-1)=0.$$

Now suppose that $s_0=1$. Let $k_0$ denote the number of $y_1, \dots, y_4, x_1, \ldots, x_4$ that are a multiple of $p$, and let $k$ denote the number of $r_1, \ldots, r_4, s_1, \ldots, s_4$ that are a multiple of $p$. If $\omega(\br,\bms, 1)\neq 0$, then we have $k \leq k_0$ and $k_0 = 4,5$ or $6$. Indeed, if $p$ divides seven of $y_1, \ldots, y_4, x_1, \ldots,x_4$ then it must also divide the eighth, which contradicts condition $(3)$ from Definition \ref{omega}, and so $k_0=7,8$ are not possible. We go through the cases $k_0=4,5,6$ in turn. For convenience we write $\omega(\br, \bms,1) = \omega(k)$ and $ \omega(\mathbf{1},\mathbf{1},1)=\omega(0)  = 1$.  

If $k_0=4$, there is one summand of (\ref{omega at primes}) with $k = 0$, namely $(\br, \bms) = (\mathbf{1}, \mathbf{1})$, and one with $k=4$, and all other terms are zero. Therefore the contribution to the sum in (\ref{omega at primes}) from $k_0=4$ is $\omega(0)+\omega(4) = 1-1 = 0$. If $k_0=  5$, then there is one summand in (\ref{omega at primes}) with $k=0$, two with $k=4$ and one with $k=5$, and so we obtain $\omega(0) + 2\omega(4) + \omega(5) = 1-2+1 = 0$. The relevant calculation for $k=6$ is 
\begin{align*}
    \omega(0) + 4\omega(4) + 4\omega(5) + \omega(6) &= 1-4+4-1 = 0.
\end{align*}
We have therefore established that (\ref{omega at primes}) holds.
\end{proof}

We now use Theorem \ref{main circle method result} to estimate $\#\mathcal{N}(B;\br, \bms,s_0)$ for each choice of $\br,\bms \in \N^4,s_0 \in \N$ satisfying $\omega(\br,\bms,s_0)\neq 0$. We have
\begin{align*}
\#\mathcal{N}(B;\br, \bms,s_0)= \sum_{\beps\in \{\pm 1\}^4}\sum_{\substack{|Y|\leq D, Y\neq \square\\ \op{sgn}y_i = \eps_i\\\br | \by}}\mu^2(y_1)\cdots \mu^2(y_4) N_{\by}(B;\bms,s_0),
\end{align*}
where 
\begin{align}  
N_{\by}(B;\bms,s_0) = \#\left\{\bx \in \N^4: 
\begin{tabular}{l}
$\sum_{i=1}^4 y_i^3 x_i^2 = 0, \bms|\bx, s_0|\bx,$\\
$|x_i| \leq \left(\frac{B}{|y_i|^3}\right)^{1/2} \textrm{ for all } i$
\end{tabular}
\right\}.
\end{align}

We make use of condition $(4)$ from Definition \ref{omega}, which allows us to make a change of variables from $x_i$ to $s_0s_ix_i$.  In the notation of (\ref{Main counting problem, with coeffs}), we have $N_{\by}(B;\bms,s_0) = N_{\bms^2\by^3}(B/s_0^2)$, where $\bms^2\by^3$ denotes the vector $(s_1^2y_1^3, \ldots, s_4^2y_4^3)$. Assuming $Y\neq \square$, we have $S^2Y^3 \neq \square$. Moreover, we note that the conditions $\omega(\br,\bms,s_0)\neq 0$ and $\br|\by$ imply that $S|Y^3$. (Indeed, if $\omega(\br,\bms,s_0)\neq 0$, then by conditions (2) and (3) of Definition \ref{omega}, $S$ must be fourth-power free. Since condition (5) implies that every prime dividing $S$ must also divide $R$, we deduce that $S|R^3$, and together with the assumption $\br|\by$, this implies that $S|Y^3$.) Therefore $|Y|\leq D$ implies that $|S^2Y^3|\leq D^9$. In particular, $|S^2Y^3|\leq B^{4/7}$ when $D\leq B^{1/16}$. Hence we may apply Theorem \ref{theorem 1.4 again} with $\ba = \bms^2\by^3$ to obtain for any $D\leq B^{1/16}$,
\begin{equation}\label{that moment when you apply a circle method result}
\#\mathcal{N}(B;\br, \bms,s_0)= \sum_{\beps\in \{\pm 1\}^4}\sum_{\substack{|Y|\leq D\\ \op{sgn}y_i = \eps_i\\\br | \by, Y \neq \square\\ y_i \textrm{ square-free}}}\left( \frac{\mathfrak{G}_{\bms^2\by^3}\sigma_{\infty}(\beps)B}{S
|Y|^{3/2}s_0^2}+ O\left(\frac{B^{41/42+\eps}\Delta^{1/3}}{|S^2Y^3|^{11/24}}\right)\right),
\end{equation}
where as in Section \ref{section the circle method}, we define
$$\Delta = \Delta(\bms^2\by^3) = \prod_{i=1}^4 \gcd\left(s_i^2y_i^3, \prod_{j\neq i}s_j^2y_j^3\right).$$
We begin by studying the main term from (\ref{that moment when you apply a circle method result}). We would like to replace the sum over $|Y|\leq D$ with a sum over all $\by \in (\Z_{\neq 0})^4$ satisfying $\op{sgn} y_i=\eps_i$ for all $i$. In order to estimate the error term in doing so, we appeal to Lemma \ref{the one and only singular series estimate}. We define
$$E_1(D) =\sum_{\substack{\br,\bms \in \N^4\\ s_0 \in \N}}\omega(\br,\bms,s_0)\sum_{\beps \in \{\pm 1\}^4}\sum_{\substack{\by \in (\Z_{\neq 0})^4\\Y > D, Y\neq \square\\ \op{sgn}y_i = \eps_i\\ \br|\by}}\frac{\mu^2(y_1)\cdots \mu^2(y_4)\mathfrak{G}_{\bms^2\by^3}\sigma_{\infty}(\beps)}{S|Y|^{3/2}s_0^2}.$$ 

\begin{lemma}\label{e1d}
For any $D\geq 1$, we have $E_1(D)=O(D^{-1/4+\eps})$. 
\end{lemma}

\begin{proof}
From the observation $S|Y^3$ made above, and the fact that $\sigma_{\infty}(\beps)\ll 1$, we have
$$E_1(D)\ll \sum_{\beps \in \{\pm 1\}^4}\sum_{\substack{\by \in \N^4\\ Y>D, Y \neq \square}}Y^{-3/2}\sum_{\substack{\bms \in \N^4\\S|Y^3}}\frac{|\mathfrak{G}_{\beps\bms^2\by^3}|}{S}\sum_{s_0\in \N}\sum_{\substack{\br \in \N^4\\ R|Y}}\frac{|\omega(\br,\bms,s_0)|}{s_0^2}.$$
Since $|\omega(\br,\bms,s_0)|\leq 1$, the sum over $s_0$ is convergent. The sum over $\br$ only contributes $O(Y^{\eps})$ by the trivial bound for the divisor function. Applying the estimate from Lemma \ref{the one and only singular series estimate}, we have $|\mathfrak{G}_{\bms^2\by^3}|\ll (SY)^{\eps}\Delta(\bms^2\by^3)^{1/4}$ whenever $Y \neq \square$. Therefore
\begin{equation}\label{estimating singular series tail}E_1(D) \ll \sum_{\substack{\by \in \N^4\\Y>D}}Y^{-3/2+\eps}\sum_{\substack{\bms \in \N^4\\ S|Y^3}}\frac{\Delta(\bms^2\by^3)^{1/4}}{S^{1-\eps}}.
\end{equation}

From Lemma \ref{basic property of delta}, we have $\Delta(\bms^2\by^3)\leq S^4\Delta(\by^3)$. Continuing from (\ref{estimating singular series tail}), we obtain the estimate
$$E_1(D)\ll \sum_{\substack{\by \in \N^4\\Y>D}}\frac{\Delta(\by^3)^{1/4}}{Y^{3/2-\eps}}\sum_{\substack{\bms \in \N^4\\ S|Y^3}}S^{\eps}\ll \sum_{\substack{\by \in \N^4\\Y>D}}\frac{\Delta(\by)^{3/4}}{Y^{3/2-\eps}}.$$

We note that $\Delta(\by)$ is always squareful. Therefore, for any $R \geq 1$, there are $O(R^{1/2})$ possible values for $\Delta(\by)$ in the range $[R,2R]$.  We define the quantity $M(\by) = Y/\Delta(\by)$. Given $M,\Delta \geq 1$, there are $O((M\Delta)^{\eps})$ choices for $\by \in \N^4$ satisfying $M=M(\by),\Delta =\Delta(\by)$. Breaking into dyadic intervals, we obtain 
\begin{align*}
    \sum_{\substack{\by \in \N^4\\Y>D}}\frac{\Delta(\by)^{3/4}}{Y^{3/2-\eps}}&\ll \sum_{\substack{R_1,R_2 \textrm{ dyadic}\\R_1R_2>D}}R_1R_2^{1/2}\max_{\substack{\by \in \N^4\\R_1\leq M(\by) \leq 2R_1\\R_2 \leq \Delta(\by) \leq 2R_2}}M(\by)^{-3/2+\eps}\Delta(\by)^{-3/4+\eps}\\
    &\ll \sum_{\substack{R_1,R_2 \textrm{ dyadic}\\R_1R_2>D}}R_1^{-1/2+\eps}R_2^{-1/4+\eps}\\
    &\ll D^{-1/4+\eps},
\end{align*}
as required.
\end{proof}

We now study the error term 
\begin{equation}\label{E2B}
E_2(D)=\sum_{\substack{\br, \bms \in \N^4\\s_0 \in \N}}\omega(\br, \bms,s_0)\sum_{\beps\in \{\pm 1\}^4}\sum_{\substack{|Y|\leq D, Y\neq \square\\ \op{sgn}y_i = \eps_i\\\br | \by,\\y_i \textrm{ square-free}}}\frac{\Delta(\bms^2\by^3)^{1/3}}{|S^2Y^3|^{11/24}}.
\end{equation}

\begin{lemma}\label{e2d}
We have $E_2(D) \ll D^{11/8+\eps}.$
\end{lemma}

\begin{proof}
We proceed in a similar fashion to the proof of Lemma \ref{e1d}. Let $M(\by)$ be as defined in Lemma \ref{e1d}. Then 

\begin{align*}
    E_2(D) &\ll \sum_{\substack{\by \in \N^4\\ Y\leq D}}\sum_{\substack{\bms \in \N^4\\ S|Y^3}}\frac{\Delta(\bms^2\by^3)^{1/3}}{(S^2Y^3)^{11/24-\eps}}\\
   &\ll \sum_{\substack{\by \in \N^4\\ Y\leq D}}\frac{\Delta(\by)}{Y^{11/8-\eps}}\sum_{\substack{\bms \in \N^4\\ S|Y^3}}S^{5/12+\eps}\\
   &\ll \sum_{\substack{R_1,R_2 \textrm{ dyadic}\\R_1R_2\leq D}}R_1R_2^{1/2}\max_{\substack{\by \in \N^4\\R_1\leq M(\by) \leq 2R_1\\R_2 \leq \Delta(\by) \leq 2R_2}}M(\by)^{-1/8+\eps}\Delta(\by)^{7/8+\eps}\\
   &\ll \sum_{\substack{R_1,R_2 \textrm{ dyadic}\\R_1R_2\leq D}}R_1^{7/8+\eps}R_2^{11/8+\eps}\\
    &\ll D^{11/8+\eps}.\qedhere
\end{align*}

\end{proof}

We now conclude the proof of Theorem \ref{the main theorem}. Combining (\ref{large error}), Lemma \ref{Mobius inversion statement}, (\ref{that moment when you apply a circle method result}), Lemma \ref{e1d} and Lemma \ref{e2d}, for any $D\leq B^{1/16}$, we have
$$N(B) = cB+O(B^{1+\eps}D^{-1/12})+O(B^{41/42+\eps}D^{11/8}),$$
where
\begin{equation}\label{the leading constant}
c = \frac{1}{16}\sum_{\beps \in \{\pm 1\}^4}\sigma_{\infty}(\beps)\sum_{\substack{\by \in (\Z_{\neq 0})^4 \\ \op{sgn}(y_i) = \eps_i\\ Y \neq \square}}\frac{\mu^2(y_1)\cdots \mu^2(y_4)}{
|Y|^{3/2}}\sum_{\substack{\br, \bms \in \N^4 \\ s_0 \in \N\\ \br|\by}}\frac{\omega(\br, \bms, s_0)\mathfrak{G}_{\bms^2\by^3}}{Ss_0^2}.
\end{equation}
Making the choice $D=B^{4/245}$, we obtain 
$N(B)= cB+ O(B^{734/735+\eps})$.

\subsection{The leading constant} The expression for the leading constant in (\ref{the leading constant}) is analogous to \cite[Equation (3.14)]{browning2019arithmetic}. Similarly to  \cite[Equation (3.15)]{browning2019arithmetic}, we now demonstrate that the inner sum of (\ref{the leading constant}) can be expressed as a product of local densities. We define for any prime $p$ and any $n \in \N$,
\begin{equation*}\label{mnyp}
	M_n(\by,p) = \#\left\{\bmm \Mod{p^n} : \sum_{i=1}^4 y_i^3 m_i^2 \equiv 0 \Mod{p^n}, p\nmid m_jy_j \textrm{ for some }j\right\}.
\end{equation*}
\begin{lemma}\label{local densities lemma} We have
$$\sum_{\substack{\br, \bms \in \N^4 \\ s_0 \in \N\\ \br|\by}}\frac{\omega(\br, \bms, s_0)\mathfrak{G}_{\bms^2\by^3}}{Ss_0^2} = \prod_{p}\lim_{N\ra \infty}\left( \frac{M_N(\by,p)}{p^{3N}}\right).$$ 
\end{lemma}

\begin{proof}
We first express the singular series $\mathfrak{G}_{\bms^2\by^3}$ as a product of local densities. Since $\mathfrak{G}_{\bms^2\by^3}=\sum_{q=1}^{\infty}S_{q,\bms^2\by^3}(\0)$, and $S_{q,\bms^2\by^3}(\0)$ is multiplicative in $q$, we have
\begin{align*}
\mathfrak{G}_{\bms^2\by^3} &= \sum_{q=1}^{\infty}q^{-4}S_{q, \bms^2\by^3}(\0)\\
&= \prod_p\left(1+\sum_{n=1}^{\infty}p^{-4n}S_{p^n,\bms^2\by^3}(\0)\right).
\end{align*}
Moreover, 
\begin{align*}
	 S_{p^n,\bms^2\by^3}(\0)&= \sum_{\bb \bmod{p^n} }\sum_{\substack{0 \leq k < p^n\\ \gcd(k,p) = 1}}e_{p^n}\left(k\sum_{i=1}^4 s_i^2y_i^3b_i^2\right)\\
    &= p^n N_{\bms, \by}(p^n) - p^{n+3}N_{\bms, \by}(p^{n-1}),
\end{align*}
where
\begin{equation*}
N_{\bms, \by}(p^n) = \#\left\{\bmm \Mod{p^n} : \sum_{i=1}^4 s_i^2 y_i^3 m_i^2 \equiv 0 \Mod{p^n}\right\}.
\end{equation*}
Therefore
\begin{align*}
\sum_{n=1}^{\infty}p^{-4n}S_{p^n,\bms^2\by^3}(\0) &= \lim_{N \ra \infty}\sum_{n=1}^N \left(p^{-3n}N_{\bms, \by}(p^n)-p^{-3(n-1)}N_{\bms,\by}(p^{n-1}) \right)\\
&= \lim_{N \ra \infty}\left( p^{-3N}N_{\bms, \by}(p^N)\right)-1,
\end{align*}
and hence
\begin{equation}
\mathfrak{G}_{\bms^2\by^3}  = \prod_p \left(\lim_{N \ra \infty}\frac{N_{\bms,\by}(p^N)}{p^{3N}} \right).
\end{equation}
For the remainder of the proof, we will use a sum over $\br^{[p]}, \bms^{[p]}, s_0^{[p]}$ to denote a sum over all $\br, \bms \in \N^4, s_0 \in \N$, with $(\br,\bms,s_0) = (\br^{[p]},\bms^{[p]},s_0^{[p]})$ and $\br|\by$. Recalling that $\omega(\br, \bms, s_0)$ is also multiplicative, we obtain
\begin{equation}\label{product of local densities}
\sum_{\substack{\br, \bms \in \N^4 \\ s_0 \in \N\\ \br|\by}}\frac{\omega(\br, \bms, s_0)\mathfrak{G}_{\bms^2\by^3}}{Ss_0^2} = \prod_p \left(\sum_{\substack{\br^{[p]}, \bms^{[p]},s_0^{[p]}}}\frac{\omega(\br^{[p]},\bms^{[p]},s_0^{[p]})}{S^{[p]}s_0^{2[p]}}\lim_{N\ra \infty}\left(\frac{N_{\bms,\by}(p^N)}{p^{3N}}\right)\right).
\end{equation}

To complete the proof, it suffices to show that for a fixed prime $p$, the factor on the right hand side of (\ref{product of local densities}) is equal to $\lim_{N \ra \infty}(p^{-3N}M_N(\by,p))$. We define 
$$N_{\bms, s_0,\by}(n)=\#\left\{\bmm \Mod{p^n}: \sum_{i=1}^4 y_i^3m_i^2 \equiv 0, \bms|\bmm, s_0|\bmm\right\}.$$ 
We claim that 
\begin{equation}\label{inclusion exclusion 2}
M_n(\by,p) = \sum_{\br^{[p]}, \bms^{[p]},s_0^{[p]}}\omega(\br, \bms, s_0)N_{\bms, s_0,\by}(n).
\end{equation}
To see this, we fix $\bmm \Mod{p^n}$ such that $\sum_{i=1}^4y_i^3m_i^2\equiv 0$, and consider the quantity 
$$\sum_{\substack{\br^{[p]}, \bms^{[p]},s_0^{[p]}\\\bms|\bmm, s_0|\bmm}}\omega(\br, \bms, s_0).$$
This expression has already been encountered in (\ref{omega at primes}), and from the proof of Lemma \ref{inclusion exclusion}, we see that it is equal to $1$ if $p|y_im_i$ for some $i$, and zero otherwise. This establishes (\ref{inclusion exclusion 2}). 

Changing variables from $m_i$ to $s_0s_im_i$, we have 
\begin{align*}
N_{\bms, s_0,\by}(n)&= \frac{s_0^{4[p]}N_{\bms,\by}(p^{n-2\nu_p(s_0)})}{S^{[p]}}.
\end{align*}
Therefore 
\begin{align*}
\lim_{N\ra \infty}\left(\frac{N_{\bms, s_0,\by}(N)}{p^{3N}}\right)&= \lim_{N\ra \infty}\left(\frac{N_{\bms,\by}(p^{N})}{S^{[p]}s_0^{2[p]}}\right).
\end{align*}
Combining this with (\ref{inclusion exclusion 2}), we deduce that $\lim_{N\ra \infty}(p^{-3N}M_N(\by,p))$ matches the Euler factor from the right hand side of (\ref{product of local densities}).
\end{proof}

We do not expect the leading constant $c$ from (\ref{the leading constant}) to agree with the the prediction from \cite[Conjecture 1.1]{pieropan2019campana}. In \cite{leadingconstant2021}, we study the counting problem $N_k(B)$ from (\ref{sum of k square-full numbers}) in the case $k=3$. This corresponds to the orbifold $(\PP^1,D)$, where $D = \frac{1}{2}[0]+\frac{1}{2}[1]+\frac{1}{2}[\infty]$. The more detailed discussion around the case $k=3$ in \cite{leadingconstant2021} is readily adapted to deal with the case $k=4$ considered in this paper. Therefore, we content ourselves with giving a brief summary here.

We recall that $N_{\by}(B)$ denotes the contribution to $N(B)$ from a fixed choice of $\by \in (\Z_{\neq 0})^4$ satisfying $Y \neq \square$ and $\mu^2(y_1) = \cdots = \mu^2(y_4) = 1$. We have shown that each $N_{\by}(B)$ contributes a positive proportion to $N(B)$, namely $N_{\by}(B) \sim c_{\by}B$, where
$$c_{\by} = \frac{\sigma_{\infty}(\beps)}{|Y|^{3/2}}\prod_p \left(\lim_{N\ra \infty}\frac{M_N(\by,p)}{p^{3N}}\right).$$
Moreover, Manin's conjecture can be applied to the quadric $Q_{\by}$ given by the equation $\sum_{i=1}^4y_i^3x_i^2 = 0$, and the constant thus predicted is in agreement with $c_{\by}$ (as was expected due to Remark \ref{manin for quadrics}). Hence the expression in (\ref{the leading constant}) is naturally interpreted as a sum over $\by$ of leading constants arising from Manin's conjecture applied to the quadrics $Q_{\by}$. This sum is not multiplicative in $\by$, and it does not appear to be possible to express (\ref{the leading constant}) as an Euler product. In contrast, the leading constant predicted in \cite[Conjecture 1.1]{pieropan2019campana} for $N(B)$ is by definition an Euler product. In \cite{leadingconstant2021}, we find that when $k=3$, the analogous constant to (\ref{the leading constant}) does not agree with \cite[Conjecture 1.1]{pieropan2019campana} numerically, and it seems very likely that the same will hold true for $k=4$. 

A natural question that arises is whether thin sets could explain a discrepancy between $c$ and the constant predicted by \cite[Conjecture 1.1]{pieropan2019campana}. However, in analogy to \cite[Theorem 1.3]{leadingconstant2021}, it can be shown that any constant in $(0,c]$ could be obtained by the removal of an appropriate thin set. Most of these thin sets have no clear geometric interpretation in relation to the original orbifold. From this point of view, the definition of thin sets of Campana points from \cite[Definition 3.7]{pieropan2019campana} seems too permissive, and $c$ seems to be the most natural choice of leading constant for the orbifold considered in this paper.

\bibliographystyle{plain}
\bibliography{references.bib}

\end{document}